\documentclass[onefignum,onetabnum]{siamart171218}

\usepackage{amsfonts,amssymb,amsmath,stmaryrd}
\usepackage{graphicx}
\usepackage{algorithmic}
\ifpdf
  \DeclareGraphicsExtensions{.eps,.pdf,.png,.jpg}
\else
  \DeclareGraphicsExtensions{.eps}
\fi

\newsiamremark{remark}{Remark}

\newsiamthm{assumption}{Assumption}

\headers{$\phi$-FEM}{M. Duprez and A. Lozinski}

\title{$\phi$-FEM: a finite element method on domains defined by level-sets}

\author{Michel Duprez\thanks{Sorbonne Universit\'es, UPMC Univ Paris 06, CNRS UMR 7598, Laboratoire Jacques-Louis Lions, F-75005, Paris, France.
  (\email{mduprez@math.cnrs.fr}, \url{http://www.imag.com/\string~ddoe/}).}
\and Alexei Lozinski\thanks{Laboratoire de Math\'ematiques de Besan\c{c}on, UMR CNRS 6623, Universit\'e Bourgogne Franche-Comt\'e,
16, route de Gray, 25030 Besan\c{c}on Cedex,
France  (\email{alexei.lozinski@univ-fcomte.fr}).}
}

\newcommand{\R}{\mathbb{R}}

\newcommand{\Th}{\mathcal{T}_h}
\DeclareMathOperator{\diam}{diam}

\newcommand{\triple}[1]{{\left\vert\kern-0.12ex\left\vert\kern-0.12ex\left\vert #1 
    \right\vert\kern-0.12ex\right\vert\kern-0.12ex\right\vert}}

\usepackage{tikz}
\usetikzlibrary{patterns}
\usepackage{pgfplots}
\pgfplotsset{compat=newest}

\newcommand{\logLogSlopeTriangle}[5]
{

    \pgfplotsextra
    {
        \pgfkeysgetvalue{/pgfplots/xmin}{\xmin}
        \pgfkeysgetvalue{/pgfplots/xmax}{\xmax}
        \pgfkeysgetvalue{/pgfplots/ymin}{\ymin}
        \pgfkeysgetvalue{/pgfplots/ymax}{\ymax}

        \pgfmathsetmacro{\xArel}{#1}
        \pgfmathsetmacro{\yArel}{#3}
        \pgfmathsetmacro{\xBrel}{#1-#2}
        \pgfmathsetmacro{\yBrel}{\yArel}
        \pgfmathsetmacro{\xCrel}{\xArel}

        \pgfmathsetmacro{\lnxB}{\xmin*(1-(#1-#2))+\xmax*(#1-#2)} 
        \pgfmathsetmacro{\lnxA}{\xmin*(1-#1)+\xmax*#1} 
        \pgfmathsetmacro{\lnyA}{\ymin*(1-#3)+\ymax*#3} 
        \pgfmathsetmacro{\lnyC}{\lnyA+#4*(\lnxA-\lnxB)}
        \pgfmathsetmacro{\yCrel}{\lnyC-\ymin)/(\ymax-\ymin)} 
        
        \coordinate (A) at (rel axis cs:\xArel,\yArel);
        \coordinate (B) at (rel axis cs:\xBrel,\yBrel);
        \coordinate (C) at (rel axis cs:\xCrel,\yCrel);

        \draw[#5]   (A)-- node[pos=0.5,anchor=north] {1}
                    (B)-- 
                    (C)-- node[pos=0.5,anchor=west] {#4}
                    cycle;
    }
}

\newcommand{\logLogSlopeTriangleinv}[5]
{

    \pgfplotsextra
    {
        \pgfkeysgetvalue{/pgfplots/xmin}{\xmin}
        \pgfkeysgetvalue{/pgfplots/xmax}{\xmax}
        \pgfkeysgetvalue{/pgfplots/ymin}{\ymin}
        \pgfkeysgetvalue{/pgfplots/ymax}{\ymax}

        \pgfmathsetmacro{\xArel}{#1}
        \pgfmathsetmacro{\yArel}{#3}
        \pgfmathsetmacro{\xBrel}{#1-#2}
        \pgfmathsetmacro{\yBrel}{\yArel}
        \pgfmathsetmacro{\xCrel}{\xBrel}

        \pgfmathsetmacro{\lnxB}{\xmin*(1-(#1-#2))+\xmax*(#1-#2)} 
        \pgfmathsetmacro{\lnxA}{\xmin*(1-#1)+\xmax*#1} 
        \pgfmathsetmacro{\lnyA}{\ymin*(1-#3)+\ymax*#3} 
        \pgfmathsetmacro{\lnyC}{\lnyA+#4*(\lnxA-\lnxB)}
        \pgfmathsetmacro{\yCrel}{(\lnyC-\ymin)/(\ymax-\ymin)} 
        
        \coordinate (A) at (rel axis cs:\xArel,\yArel);
        \coordinate (B) at (rel axis cs:\xBrel,\yBrel);
        \coordinate (C) at (rel axis cs:\xCrel,\yCrel);

        \draw[#5]   (A)-- node[pos=0.5,anchor=north] {1}
                    (B)-- node[pos=0.5,anchor=east] {#4}
                    (C)-- 
                    cycle;
    }
}

\ifpdf
\hypersetup{
  pdftitle={phi-FEM},
  pdfauthor={M. Duprez and A. Lozinski}
}
\fi

\begin{document}

\maketitle

\begin{abstract}
We propose a new fictitious domain finite element method, well suited for elliptic problems posed in a domain given by a level-set function without requiring a mesh fitting the boundary. To impose the Dirichlet boundary conditions, we search the approximation to the solution as a product of a finite element function with the given level-set function, which also approximated by finite elements.
Unlike other recent fictitious domain-type methods (XFEM, CutFEM), our approach does not need any non-standard numerical integration (on cut mesh elements or on the actual boundary).
We consider the Poisson equation discretized with piecewise  polynomial  Lagrange  finite  elements of any order and prove the optimal convergence of our method in the $H^1$-norm. Moreover, the discrete problem is proven to be well conditioned, \textit{i.e.} the condition number of the associated finite element matrix is of the same order as that of a standard finite element method on a comparable conforming mesh. Numerical results confirm the optimal convergence in both $H^1$ and $L^2$ norms.
\end{abstract}



\section{Introduction}
We consider the Poisson-Dirichlet problem
\begin{equation}
  \label{eq:poisson} \left\{ \begin{array}{cl}
    - \Delta u = f & \mathrm{on} \quad \Omega,\\
    u = 0 & \mathrm{on} \quad \Gamma,
  \end{array} \right.
\end{equation}
in a bounded domain $\Omega\subset\R^d$ ($d=2,3$) with smooth boundary $\Gamma$ assuming that $\Omega$ and $\Gamma$ are given by a level-set function $\phi$:
\begin{equation}\label{eq:level-set}
 \Omega:=\{\phi<0\} \text{ and } \Gamma:=\{\phi=0\}.
\end{equation} 
Such a representation is a popular and useful tool to deal with problems with evolving surfaces or interfaces \cite{osher}.
In the present article, the level-set function is supposed known on $\R^d$, smooth, and to behave near $\Gamma$ as the signed distance to $\Gamma$. We propose a finite element method for the problem above which is easy to implement, does not require a mesh fitted to $\Gamma$, and is guaranteed to converge optimally. Our basic idea is very simple: one cannot impose the Dirichlet boundary conditions in the usual manner since the boundary $\Gamma$ is not resolved by the mesh, but one can search the approximation to $u$ as a product of a finite element function $w_h$ with the level-set $\phi$ itself: such a product obviously vanishes on $\Gamma$. In order to make this idea work, some stabilization should be added to the scheme as outlined below and explained in detail in the next section. We coin our method $\phi$-FEM in accordance with the tradition of denoting the level-sets by $\phi$.

More specifically, let us assume that $\Omega$ lies inside a simply shaped domain $\mathcal{O}$ (typically a box in $\R^d$) and introduce a quasi-uniform simplicial mesh $\Th^{\mathcal{O}}$ on $\mathcal{O}$ (the background mesh). Let $\Th$ be a submesh of $\Th^{\mathcal{O}}$ obtained by getting rid of mesh elements lying entirely outside $\Omega$ (the definition of $\Th$ will be slightly changed afterwords). Denote by $\Omega_h$ the domain covered by the mesh $\Th$ (so that typically $\Omega_h$ is only slightly larger than $\Omega$). Our starting point is the following formal observation: assuming that the right-hand side $f$ is actually well defined on $\Omega_h$, and the solution $u$ can be extended to $\Omega_h$ so that $- \Delta u = f$ on $\Omega_h$, we can introduce the new unknown $w\in H^1(\Omega_h)$ such that $u=\phi w$ and the boundary condition on $\Gamma$ is 	automatically satisfied. An integration by parts yields then
\begin{equation}\label{wphi}
\int_{\Omega_h} \nabla (\phi w) \cdot \nabla (\phi v) -
\int_{\partial \Omega_h} \frac{\partial}{\partial n}  (\phi w)
\phi v = \int_{\Omega_h} f \phi v, \quad\forall v\in H^1(\Omega_h).
\end{equation}
Given a finite element approximation $\phi_h$ to $\phi$ on the mesh  $\Th$ and a finite element space $V_h$ on $\Th$, one can then try to search for $w_h\in V_h$ such that the equality in (\ref{wphi}) with the subscripts $h$ everywhere is satisfied for all the test functions $v_h\in V_h$ and to reconstruct an approximate solution $u_h$ to (\ref{eq:poisson}) as $\phi_h w_h$. These considerations are very formal and, not surprisingly, such a method does not work as is. We shall show however that it becomes a valid scheme once a proper stabilization in the vein of the Ghost penalty \cite{burmanghost} is added. The details on the stabilization and on the resulting finite element scheme are given in the next section. 

Our method shares many features with other finite elements methods on non-matching meshes, such as XFEM  \cite{moes99,moes06,sukumar2001,HaslingerRenard} or CutFEM \cite{burman1,burman2,burman3,burman15}. Unlike the present work, the integrals over $\Omega$ are kept in XFEM or CutFEM discretizations, which is cumbersome in practice since one needs to implement the integration on the boundary $\Gamma$ and on parts of mesh elements cut by the boundary. The first attempt to alleviate this practical difficulty was done in \cite{lozinski2018} with method that does not require to perform the integration on the cut elements, but needs still the integration on $\Gamma$. In the present article, we fully avoid any non trivial numerical integration: all the integration in $\phi$-FEM is performed on the whole mesh elements, and there are no integrals on $\Gamma$. We also note that an easily implementable version of $\phi$-FEM is here developed for $P_k$ finite elements of any order $k\ge 1$. This should be contrasted with the situation in CutFEM where some additional terms should be added in order to achieve the optimal $P_k$ accuracy if $k>1$, cf. \cite{burman18}.

The article is structured as follows: our $\phi$-FEM method is presented in the next section. We also give there the assumptions on the level-set $\phi$ and on the mesh, and announce our main result: the \textit{a priori} error estimate for $\phi$-FEM. We work with standard continuous $P_k$ finite elements on a simplicial mesh and prove the optimal order $h^k$ for the error  in the $H^1$ norm and the (slightly) suboptimal order $h^{k+1/2}$ for the error  in the $L^2$ norm. The proofs of these estimates are the subject of Section 3. Moreover, we prove in Section 4 that the associated finite element matrix has the condition number of order $1/h^2$, the same as that of a standard finite element method.  Some numerical illustrations are given in Section 5. 

\section{Definitions, assumptions, description of $\phi$-FEM, and the main result}
We recall that we work with a bounded domain $\Omega\subset\mathcal{O}\subset\R^d$ ($d=2,3$) with boundary $\Gamma$ given by a level-set $\phi$ as in (\ref{eq:level-set}). We assume that $\phi$ is sufficiently smooth and behaves near $\Gamma$ as the signed distance to $\Gamma$ after an appropriate change of local coordinates. More specifically, we fix an integer $k\ge 1$ and introduce the following
\begin{assumption}\label{asm0}
The boundary $\Gamma$ can be covered by open sets $\mathcal{O}_i$, $i=1,\ldots,I$ and one can introduce on every $\mathcal{O}_i$ local coordinates $\xi_1,\ldots,\xi_d$ with $\xi_d=\phi$ such that all the partial derivatives $\partial^\alpha\xi/\partial x^\alpha$ and $\partial^\alpha x/\partial \xi^\alpha$ up to order $k+1$ are bounded by some $C_0>0$. Morover, $|\phi|\ge m$ on $\mathcal{O}\setminus\cup_{i=1,\ldots,I}\mathcal{O}_i$ with some $m>0$.
\end{assumption}

Let $\Th^{\mathcal{O}}$ be a quasi-uniform simplicial mesh on $\mathcal{O}$ of mesh size $h$, meaning that $\diam(T)\le h$ and $\rho(T)\ge \beta h$ for all simplexes $T\in\Th^{\mathcal{O}}$ with some mesh regularity parameter $\beta>0$ ($\rho(T)$ stands for the radius of the largest ball inscribed in $T$). Consider, for an integer $l\ge 1$, the finite element space
 \begin{eqnarray*}
  V_{h,\mathcal{O}}^{(l)} = \{v_h \in H^1 (\mathcal{O}) : v_h |_T \in \mathbb{P}_l (T)
  \  \forall ~T \in \mathcal{T}_h^{\mathcal{O}} \}.
\end{eqnarray*}
Introduce an approximate level-set $\phi_h\in V_{h,\mathcal{O}}^{(l)}$ by
\begin{equation}\label{eq:phi_h}
\phi_h:=I_{h,\mathcal{O}}^{(l)}(\phi)
\end{equation}
 where $I_{h,\mathcal{O}}^{(l)}$ is the standard Lagrange interpolation operator on $V_{h,\mathcal{O}}^{(l)}$. We shall use this to approximate the physical domain $\Omega=\{\phi<0\}$ with smooth boundary $\Gamma=\{\phi=0\}$ by the domain $\{\phi_h<0\}$ with the piecewise polynomial boundary $\Gamma_h=\{\phi_h=0\}$. We employ $\phi_h$ rather than $\phi$ in our numerical method in order to simplify its implementation (all the integrals in the forthcoming finite element formulation will involve only the piecewise polynomials).   This feature will also turn out crucial in our theoretical analysis.

 We now introduce the computational mesh $\Th$ as the subset of $\Th^{\mathcal{O}}$ composed of the triangles/tetrahedrons having a non-empty intersection with the approximate domain $\{\phi_h<0\}$. We denote the domain occupied by $\Th$ by $\Omega_h$, \textit{i.e.}
\[\Th:=\{T\in \Th^{\mathcal{O}}:T\cap \{\phi_h<0\}\neq  \varnothing\}
\mbox{~~~~~and~~~~~} {\Omega_h} = (\cup_{T \in \Th}T)^o. \]
\begin{remark}
Note that we do not necessarily have $\Omega\subset\Omega_h$. Indeed some mesh elements can be cut by the exact boundary  $\{\phi=0\}$ but not with the approximate one  $\{\phi_h=0\}$. Such a mesh element will not be part of $\Th$ although it contains a small portion of $\Omega$.
\end{remark}

Fix an integer $k\ge 1$ (the same $k$ as in Assumption \ref{asm0}) and consider the finite element space 
\begin{eqnarray*}
  V_h^{(k)} = \{v_h \in H^1 (\Omega_h) : v_h |_T \in \mathbb{P}_k (T)
  \  \forall ~T \in \mathcal{T}_h \}.
\end{eqnarray*}
The $\phi$-FEM approximation to \eqref{eq:poisson} is introduced as follows: 
find $w_h\in V_h^{(k)}$ such that:
\begin{equation}\label{eq:discret prob}
a_h(w_h,v_h)=l_h(v_h)\mbox{ for all }v_h\in V_h^{(k)},
\end{equation}
where the bilinear form $a_h$ and the linear form $l_h$ are defined by
\begin{equation}\label{ah}
   a_h(w,v):=\displaystyle \int_{\Omega_h} \nabla (\phi_h w) \cdot \nabla (\phi_h v) -
\int_{\partial \Omega_h} \frac{\partial}{\partial n} (\phi_h w) \phi_h
   v + G_h (w, v) 
\end{equation}	
   and
   \[\displaystyle   l_h(v):= \int_{\Omega_h} f \phi_h v+ G_h^{rhs} (v), \]
with $G_h$ and $G_h^{rhs}$ standing for
\[ G_h(w, v): = \displaystyle\sigma h\sum_{E\in  \mathcal{F}_h^{\Gamma}} 
\int_E \left[ \frac{\partial}{\partial n}
   (\phi_h w) \right] \left[ \frac{\partial}{\partial n} ( \phi_hv)\right]
   + \sigma h^2 \sum_{T \in\Th^{\Gamma}} \int_T \Delta(\phi_h w) \Delta(\phi_h v)  \]
   and
\[   G_h^{rhs} (v): =\displaystyle- \sigma h^2\sum_{T \in\Th^{\Gamma}} \int_T f \Delta (\phi_h v)   \]
where $\sigma> 0$ is an $h$-independent stabilization parameter,  $\Th^{\Gamma}\subset\Th$ contains the mesh elements cut by the approximate boundary $\Gamma_h=\{\phi_h=0\}$, 
\textit{i.e.}
\begin{equation}\label{eq:def ThGamma}
\Th^{\Gamma}=\{T\in\Th:T\cap\Gamma_h\neq\varnothing\},
\quad
\Omega_h^{\Gamma}:=\left(\cup_{T\in \mathcal{T}_h^{\Gamma}}T\right)^o.
\end{equation}
and $\mathcal{F}_h^{\Gamma}$ collects the interior facets of the mesh $\Th$ either cut by $\Gamma_h$ or  belonging to a cut mesh element
\begin{equation*}
  \mathcal{F}_h^{\Gamma} = \{E \text{ (an internal facet of } \mathcal{T}_h)
  \text{ such that } \exists T \in \mathcal{T}_h : T \cap \Gamma_h \neq
  \varnothing \text{ and } E \in \partial T\}.
\end{equation*}
The brackets inside the integral over $E\in\mathcal{F}_h^{\Gamma}$ in the formula for $G_h$ stand for the jump over the facet $E$. 

\begin{remark}
The term $G_h$ in $a_h$ is the stabilization which differentiate the method introduced here from its naive version (\ref{wphi}) from the Introduction. The first part in $G_h$ actually coincides with the ghost penalty as introduced in \cite{burmanghost} for $P_1$ finite elements. We add here another term involving the laplacian of $\phi_hw_h$. To make the stabilization consistent, this term is compensated by yet another term on the right-hand side -- $G_h^{rhs}$. Indeed, $\phi_hw_h$ should approximate the exact solution $u$ and $-\Delta u=f$. We shall show that such a stabilization makes the bilinear form $a_h$ coercive on $P_k$ finite elements of any order $k\ge 1$. Note that the usual choice for the ghost stabilization in the CutFEM literature is more complicated in the case of $P_k$ elements, $k>1$, cf \cite{burman3}: it involves the jumps of higher order normal derivatives up to the order $k$. We believe that our additional stabilization with the laplacians could be used in the CutFEM context as well. In this way, one would avoid the derivatives of order $>2$ even on polynomials of degree $k>2$ making the implementation somewhat simpler.   
\end{remark}

We shall also need the following assumptions on the mesh $\Th$, more specifically on the intersection of elements of $\Th$ with the approximate boundary $\Gamma_h=\{\phi_h=0\}$. This assumption is normally satisfied for $h$ small enough, cf. the discussion in \cite{lozinski2018}.

\begin{assumption}\label{asm2}
  The approximate boundary $\Gamma_h$ can be covered by element patch\-es $\{\Pi_i \}_{i = 1,
  \ldots, N_{\Pi}}$ having the following properties:
  \begin{itemize}
    \item Each patch $\Pi_i$ is a connected set composed of a mesh element $T_i\in\Th\setminus\Th^{\Gamma}$ and some mesh elements cut by $\Gamma_h$. More precisely, $\Pi_i = T_i \cup \Pi_i^{\Gamma}$ with $\Pi_i^{\Gamma}\subset\mathcal{T}_h^{\Gamma}$ containing at most $M$ mesh elements;    
    \item $\mathcal{T}_h^{\Gamma} = \cup_{i = 1}^{N_{\Pi}} \Pi_i^{\Gamma}$;
    \item $\Pi_i$ and $\Pi_j$ are disjoint if $i \neq j$.
  \end{itemize}
\end{assumption}

In what follows, $\|\cdot\|_{k,\mathcal{D}}$ (resp. $|\cdot|_{k,\mathcal{D}}$) denote the norm (resp. the semi-norm) in the Sobolev space $H^k(\mathcal{D})$ with an integer $k\ge 0$ where $\mathcal{D}$ can be a domain in $\mathbb{R}^d$ or a $(d-1)$-dimensional manifold. 

\begin{theorem}\label{th:error}
Suppose that Assumptions \ref{asm0} and \ref{asm2} hold true, $l\ge k$, the mesh $\Th$ is quasi-uniform, and $f\in H^k(\Omega_h\cup\Omega)$.
Let $u\in H^{k+2}(\Omega)$  be the solution to \eqref{eq:poisson} and $w_h\in V_h^{(k)}$ be the solution to \eqref{eq:discret prob}.  Denoting $u_h:=\phi_h w_h $, it holds
\begin{equation}\label{H1err}
  | u - u_h|_{1, \Omega\cap\Omega_h} \le Ch^k \|f \|_{k, \Omega\cup\Omega_h}
\end{equation}	
with a constant $C>0$ depending on the constants in Assumptions \ref{asm0}, \ref{asm2} (and thus depending on the regularity of $\phi$), and on the mesh regularity, but independent of $h$, $f$, and $u$. Moreover, supposing $\Omega\subset\Omega_h$	
\begin{equation}\label{L2err}
   \| u - u_h\|_{0, \Omega} \le Ch^{k+1/2} \|f \|_{k, \Omega_h}
\end{equation}		
with a constant $C>0$ of the same type.	
\end{theorem}

\section{Proof of the \textit{a priori} error estimate}

This section is devoted to the proof of Theorem \ref{th:error}. We first give some preliminary results, starting with a Hardy-type inequality which will allow us to properly introduce the new unknown $w=u/\phi$. This will be followed by some technical lemmas, mostly about the properties of functions of the form $\phi_hv_h$ with $v_h\in V_h^{(k)}$.

\subsection{A Hardy-type inequality}

\begin{lemma}
  \label{lemma:hardy}We assume that the domain $\Omega$ is given by the
  level-set $\phi$, cf. (\ref{eq:level-set}), and satisfies Assumption \ref{asm0}. Then, for any $u \in H^{k + 1} (\mathcal{O})$ vanishing on $\Gamma$,
  \[ \left\| \frac{u}{\phi} \right\|_{k, \mathcal{O}} \le C \| u
     \|_{k + 1, \mathcal{O}} \]
  with $C > 0$ depending only on the constants in Assumption \ref{asm0}. 
\end{lemma}

\begin{proof}
  The proof is decomposed into three steps:
  
  \textbf{Step 1. }We start in the one dimensional setting and adapt the proof oF Hardy's inequality from \cite{HardyBlog}. Let $u : \mathbb{R} \rightarrow \mathbb{R}$ be a
  $C^{\infty}$ function with compact support such that $u (0) = 0$. Set $w (x)
  = u (x) / x$ for $x \neq 0$. We shall prove that $w$ can be extended to a
  $C^{\infty}$ function on $\mathbb{R}$ and, for any \ integer $s \ge
  0$,
  \begin{equation}
    \label{Hardys} \left( \int_{- \infty}^{\infty} |w^{(s)} (x) |^2 
    \hspace{0.17em} dx \right)^{1 / 2} \le C \left( \int_{- \infty}^{\infty}
    |u^{(s + 1)} (x) |^2  \hspace{0.17em} dx \right)^{1 / 2}
  \end{equation}
  with $C$ depending only on $s$.
  
  Observe, for any $x > 0$,
  \[ w (x) = \frac{u(x)}{x}= \frac{1}{x}\int_0^x u' (t) dt = \int_0^1 u'  (xt) dt. \]
  Hence
  \begin{equation}
    \label{ws} w^{(s)} (x) = \int_0^1 u^{(s + 1)}  (xt) t^s dt.
  \end{equation}
  We have now by the integral version of Minkowski's inequality
    \begin{multline*}
    \left( \int_0^{\infty} |w^{(s)} (x) |^2  \hspace{0.17em} dx \right)^{1 /
    2} = \left( \int_0^{\infty} \left| \int_0^1 u^{(s + 1)} (xt) t^s dt
    \right|^2  \hspace{0.17em} dx \right)^{1 / 2} \\
		 \le \int_0^1 \left(\int_0^{\infty} |u^{(s + 1)} (xt) |^2  \hspace{0.17em} dx \right)^{1 / 2} 
    \hspace{0.17em} t^s dt
   = C \left( \int_0^{\infty} |u^{(s + 1)} (x) |^2  \hspace{0.17em} dx
     \right)^{1 / 2} 
		\end{multline*}
  with $C =\int_0^1 t^{s - 1 / 2} dt = 1 / (s + 1 / 2)$. Applying the same argument to negative $x$ we
  also have
  \[ \left( \int_{- \infty}^0 |w^{(s)} (x) |^2  \hspace{0.17em} dx \right)^{1
     / 2} \le C \left( \int_{- \infty}^0 |u^{(s + 1)} (x) |^2  \hspace{0.17em}
     dx \right)^{1 / 2} .\]
  Adding this to the preceding bound on $(0, + \infty)$ we get (\ref{Hardys})
  assuming that $w^{(s)}$ is continuous at $x = 0$. To prove this last point,
  we pass to the limit $x \rightarrow 0^+$ in (\ref{ws}) to see that
  $\lim_{x \rightarrow 0^+} w^{(s)} (x) = u^{(s + 1)} (0) /(s+1)$. 
  The same formula holds for the limit as $x \rightarrow 0^-$. This means that
  $w$ is continuous if we define $w (0) = u' (0)$ and $w^{(s)} (0)$ exists for
  all $s$.
  
  \textbf{Step 2.} Let now $u : \mathbb{R}^d \rightarrow \mathbb{R}$ be a
  compactly supported $C^{\infty}$ function vanishing at $x_d = 0$ and set $w
  = u / x_d$. We shall prove
  \begin{equation}
    \label{HardyRn} |w|_{k, \mathbb{R}^d} \le C |u|_{k + 1, \mathbb{R}^d}
  \end{equation}
  with $C$ depending only on $k$.
  
  To keep things simple, we give here the proof for the case $d = 2$ only (the
  case $d = 3$ is similar but would involve more complicated notations). Take
  any integers $t, s \ge 0$ with $t + s = k$, apply (\ref{Hardys}) to
  $\frac{\partial^t w}{\partial x_1^t} = \frac{1}{x_2}  \frac{\partial^t
  u}{\partial x_1^t}$ treated as a function of $x_2$ (note that
  $\frac{\partial^t u}{\partial x_1^t}$ vanishes at $x_2 = 0$) and then
  integrate with respect to $x_1$. This gives
  \[ \left\| \frac{\partial^k w}{\partial x_1^t \partial x_2^s} \right\|_{0,
     \mathbb{R}^d} \le C \left\| \frac{\partial^{k + 1} u}{\partial x_1^t \partial
     x_2^{s + 1}} \right\|_{0, \mathbb{R}^d} . \]
  Thus,
  \[ |w|_{k, \mathbb{R}^d}^2 = \sum_{s = 0}^k \left\| \frac{\partial^k w}{\partial
     x_1^{k - s} \partial x_2^s} \right\|_{0, \mathbb{R}^d}^2 \le C^2  \sum_{s
     = 0}^k \left\| \frac{\partial^{k + 1} u}{\partial x_1^{k - s} \partial x_2^{s +
     1}} \right\|_{0, \mathbb{R}^d}^2 \le C^2 |u|_{k + 1, \mathbb{R}^d}^2 \]
  so that (\ref{HardyRn}) is proved.
  
  \textbf{Step 3.} Consider finally the domains $\Omega \subset \mathcal{O}$
  as announced in the statement of this Lemma, let $u$ be a $C^{\infty}$
  function on $\mathcal{O}$ vanishing on $\Gamma$, and set $w = u / \phi$.
  Assume first that $u$ is compactly supported in $\mathcal{O}_l$, one of the
  sets forming the cover of $\Gamma$ as announced in Assumption \ref{asm0}. Recall
  the local coordinated $\xi_1, \ldots, \xi_d$ on $\mathcal{O}_l$ with $\xi_d
  = \phi$ and denote by $\hat{u}$ (resp. $\hat{w}$) the function $u$ (resp.
  $w$) treated as a function of $\xi_1, \ldots, \xi_d$. Since $\hat{w} =
  \hat{u} / \xi_d$, (\ref{HardyRn}) implies $\| \hat{w} \|_{k, \mathbb{R}^d}
  \le C \| \hat{u} \|_{k + 1, \mathbb{R}^d}$. Passing from the
  coordinates $x_1, \ldots, x_d$ to $\xi_1, \ldots, \xi_d$ and backwards we
  conclude $\| w \|_{k, \mathcal{O}_l} \le C \| u \|_{k + 1,
  \mathcal{O}_l}$ with a constant $C$ that depends on the maximum of partial
  derivatives $\partial^{\alpha} x / \partial \xi^{\alpha}$ up to order $k$
  and that of $\partial^{\alpha} \xi / \partial x^{\alpha}$ up to order $k+1$.
  Introducing a partition of unity subject to the cover $\{ \mathcal{O}_l \}$
  we can now easily prove $\| w \|_{k, \mathcal{O}} \le C \| u \|_{k +
  1, \mathcal{O}}$ noting that $1 / \phi$ is bounded outside $\cup_l \{
  \mathcal{O}_l \}$. This estimate holds also true for $u \in H^{k + 1}
  (\mathcal{O})$ by density of $C^{\infty}$ in $H^{k + 1}$.
\end{proof}
\begin{remark}\label{RemAsm0} Assumption \ref{asm0} used in the lemma above implies in particular that $\phi$ is of class $C^{k+1}$, and the constant $C_0$ form this Assumption serves as an upper bound for the norm of $\phi$ in $C^{k+1}$. Note that, this can be relaxed. For example, in the case $k=0$, it suffices to require that $\phi$ is in $W^{1,\infty}$. In particular, $\phi$ can be a continuous piecewise polynomial function with its gradient bounded almost everywhere by $C_0$.   
\end{remark}

\subsection{Some technical lemmas}

\begin{lemma}\label{lemma:poly}
Let $T$ be a triangle/tetrahedron, $E$ one of its
sides and $p$ a polynomial on $T$ of degree $l\ge 0$ such that $p=\frac{\partial
p}{\partial n}=0$ on $E$ and $\Delta p = 0$ on $T$. Then $p = 0$ on $T$.
\end{lemma}
\begin{proof} Let us consider only the 2D case (3D is similar). Without loss of
generality, we can assume that $E$ lies on on the $x$-axis in $(x, y)$
coordinates. Let $p = \sum p_{{ij}} x^i y^j$ with $i, j \ge 0$, $i + j \le l$ as above. We shall prove by induction on $m = 0, 1, \ldots, l$ that $p_{i m} = 0$, $\forall i$.
Indeed, this is valid for $m=0,1$ since $p (x, 0) = \sum_i p_{i 0} x^i = 0$ and $\frac{\partial p}{\partial y} (x, 0) = \sum_i p_{i 1} x^i = 0$. Now, $\Delta p = 0$ implies for all indices $i, j \ge 0$
\[ (i+2)(i+1)p_{i + 2, j} + (j+2)(j+1)p_{i, j + 2} = 0 \]
so that $p_{i m} = 0$, $\forall i$ implies $p_{i,m+2} = 0$, $\forall i$. 
\end{proof}

Recall the definition of the submesh $\Th^{\Gamma}$ and introduce the correspoinding domain $\Omega_h^{\Gamma}$ (\ref{eq:def ThGamma}).
The following lemma extends a similar result from \cite{lozinski2018} where it was proved for piecewise linear finite elements. 

\begin{lemma}
  \label{LemDir:prop1}
Under Assumption \ref{asm2},   for any $\beta > 0$ and $s\in\mathbb{N}^*$ one can choose $0 < \alpha < 1$
  depending only on the mesh regularity and $s$ such that, for each $v_h\in V_h^{(s)}$,
  \begin{equation}
    |v_h |_{1, \Omega_h^{\Gamma}}^2 \le \alpha |v_h |_{1, \Omega_h}^2 +
    \beta h\sum_{E \in \mathcal{F}_h^{\Gamma}} \left\| \left[ \frac{\partial
    v_h}{\partial n} \right] \right\|_{0, E}^2 + \beta h^2 \sum_{T \in
    \Th^{\Gamma}} \|\Delta v_h\|_{0, T}^2. \label{Dir:prop1}
  \end{equation}
\end{lemma}

\begin{proof}
  Choose any $\beta > 0$, consider the decomposition of $\Omega_h^\Gamma$ in element patches $\{\Pi_k \}$ as in Assumption \ref{asm2}, and introduce
  \begin{equation}
    \label{Dir:alph} \alpha := \max_{\Pi_k, v_h \neq 0}  \frac{|v_h |_{1,
    \Pi_k^{\Gamma}}^2 - \beta h \sum_{E \in \mathcal{F}_k} \left\| \left[
    \frac{\partial v_h}{\partial n} \right] \right\|_{0, E}^2
-\beta h^2\sum_{T\subset\Pi_k} \|\Delta v_h\|_{0,T}^2    
    }{|v_h |_{1,\Pi_k}^2},
    \end{equation}
  where the maximum is taken over all the possible configurations of a patch
  $\Pi_k$ allowed by the mesh regularity and over all the piecewise polynomial
  functions on $\Pi_k$ (polynomials of degree $\le s$). The subset $\mathcal{F}_k \subset
  \mathcal{F}_h^{\Gamma}$ gathers the edges internal to $\Pi_k$. Note that the
  quantity under the $\max$ sign in \eqref{Dir:alph} is invariant under the
  scaling transformation $x \mapsto hx$ and is homogeneous with respect to
  $v_h$. Recall also that the patch $\Pi_k$ contains a most $M$ elements. Thus, the maximum is indeed attained since it is taken over a bounded
  set in a finite dimensional space. 
	
	Clearly, $\alpha \le 1$. Supposing
  $\alpha = 1$ would lead to a contradiction. Indeed, if $\alpha = 1$ then we
  can take $\Pi_k$, $v_h$ yielding this maximum and suppose without loss of
  generality $|v_h |_{1, \Pi_k} = 1$. We observe then
  \begin{eqnarray*}
    |v_h |_{1, T_k}^2 + \beta h\sum_{E \in \mathcal{F}_k}  \left\| \left[
    \frac{\partial v_h}{\partial n} \right] \right\|_{0, E}^2 +\beta h^2\sum_{T\subset\Pi_k} \|\Delta v_h\|_{0,T}^2= 0
  \end{eqnarray*}
	since $|v_h |_{1,\Pi_k}^2=|v_h |_{1,T_k}^2+|v_h |_{1,\Pi_k^\Gamma}^2$.
  This implies $v_h = c=const$ on $T_k$, $\left[\frac{\partial v_h}{\partial n} \right] = 0$ on all $E \in
  \mathcal{F}_k$, and $\Delta v_h=0$ on all $T\subset\Pi_k$.
 Thus applying Lemma \ref{lemma:poly} to $v_h-c$, we deduce that $v_h=c$ on  $\Pi_k$, which contradicts $|v_h
  |_{1, \Pi_k} = 1$. 
	
	This proves $\alpha < 1$. We have thus
  \begin{eqnarray*}
    |v_h |_{1, \Pi_k^{\Gamma}}^2 \le \alpha |v_h |_{1, \Pi_k}^2 + \beta 
    h\sum_{E \in \mathcal{F}_k} \left\| \left[ \frac{\partial v_h}{\partial n}
    \right] \right\|_{0, E}^2+\beta h^2\sum_{T\subset\Pi_k} \|\Delta v_h\|_{0,T}^2
  \end{eqnarray*}
  for all $v_h\in V_h$ and all the admissible patches $\Pi_k$. Summing this over $\Pi_k$,
  $k = 1, \ldots, N_{\Pi}$ yields (\ref{Dir:prop1}).
\end{proof}

\begin{lemma}\label{Prelim1}
For all $v_h \in V_h^{(k)}$
\begin{equation}\label{Dir:eq:3}
\left\|\phi_h v_h \right\|_{0,\Omega_h^{\Gamma}}
\le  Ch\left |\phi_h v_h \right|_{1, \Omega_h^{\Gamma}}
\end{equation}  
with a constant $C>0$ depending only on the regularity of $\Th$.
\end{lemma}
\begin{proof}
Take any $T \in \mathcal{T}^{\Gamma}_h$ and let $p_h = \phi_h v_h$ on $T$.
This is a polynomial in $\mathbb{P}_{k + l}$ vanishing at at least one point
of \ $T$. We want to prove
\begin{equation}
  \label{PoinLoc} \| p_h \|_{0, T} \le Ch_T | p_h |_{1, T}
\end{equation}
with $h_T = \diam(T)$, which would entail (\ref{Dir:eq:3}) by summing
over all $T \in \mathcal{T}^{\Gamma}_h$. To prove (\ref{PoinLoc}), we consider
the following supremum
\begin{equation}
  \label{PoinSup} C = \sup_{p_h \neq 0, T} \frac{\| p_h \|_{0, T}}{h_T | p_h
  |_{1, T}}
\end{equation}
taking over all the polynomials in $\mathbb{P}_{k + l}$ vanishing at a point
of \ $T$ and all the simplexes $T$ satisfying the regularity assumption $h_T /
\rho (T) \ge \beta$. Note that the denominator in (\ref{PoinSup}) never
vanishes if $p_h \neq 0$. Indeed, $| p_h |_{1, T} = 0$ would imply $p_h = 0$
since $p_h$ vanishes at a point. By homogeneity, the supremum in
(\ref{PoinSup}) can be restricted to $p_h$ with $\| p_h \|_{0, T} =1$
and to simplexes $T$ with $h_T = 1$. This supremum is thus taken over a closed
bounded set in a finite dimensional space so that it is attained. This means that $C$
is finite which entails (\ref{PoinLoc}) and (\ref{Dir:eq:3}).
\end{proof}

\begin{remark}\label{RemPrelim1}
Inequality (\ref{Dir:eq:3}) is also valid on $\Omega_h\setminus\Omega$ instead of $\Omega_h^\Gamma$. Typically, we have any way $\Omega_h\setminus\Omega\subset\Omega_h^\Gamma$. But it can happen that the real boundary $\Gamma$ goes slightly outside of $\Omega_h^\Gamma$, which is defined by intersections with $\Gamma_h$. To deal with this situations, we can add more neighbor mesh elements into $\Omega_h^\Gamma$ and prove
\begin{equation}\label{Dir:eq:3alt}
\left\|\phi_h v_h \right\|_{0,\Omega_h\setminus\Omega}
\le  Ch\left |\phi_h v_h \right|_{1, \Omega_h}
\end{equation}  
\end{remark}

\begin{lemma}\label{Prelim2}
For all $v_h \in V_h^{(k)}$
  \begin{equation}
    \sum_{E \in \mathcal{F}_h^{\Gamma}} \|\phi_hv_h \|_{0, E}^2 \le C  h|\phi_hv_h |_{1,\Omega_h}^2 \label{Dir:prop2}
  \end{equation}
and
  \begin{equation}
    \|\phi_hv_h \|_{0, \partial\Omega_h}^2 \le C  h|\phi_hv_h |_{1,\Omega_h}^2 \label{Dir:prop2a}
  \end{equation}
with a constant $C>0$ depending only on the regularity of $\Th$.
\end{lemma}

\begin{proof}
  Let $E\in \mathcal{F}_h^{\Gamma}$. Recall the well-known trace inequality
  \begin{equation}\label{eq:trace}
    \|v\|_{0, E}^2 \le C \left( \frac{1}{h} \|v\|_{0, T}^2 + h|v|_{1, T}^2
    \right) %
  \end{equation}
  for each $v\in H^1( E)$.
  Summing   this over all $E \in \mathcal{F}_h^{\Gamma}$ gives
	$$
	  \sum_{E \in \mathcal{F}_h^{\Gamma}} \|\phi_hv_h \|_{0, E}^2 \le C 
		\left( \frac{1}{h} \|\phi_hv_h\|_{0, \Omega_h^{\Gamma}}^2 + h|\phi_hv_h|_{1, \Omega_h^{\Gamma}}^2 \right) 
	$$
	leading, in combination with  (\ref{Dir:eq:3}), to (\ref{Dir:prop2}). The proof of (\ref{Dir:prop2a}) is similar.
\end{proof}

\begin{lemma}\label{Prelim15}
Under Assumption \ref{asm0}, it  holds for all $v \in H^{s} (\Omega_h)$ with integer $1\le s\le k+1$,
$v$ vanishing on  $\Omega$,
\begin{equation}\label{Dir:eq:4}
\left\|v \right\|_{0,\Omega_h\setminus\Omega}
\le  Ch^s\left\|v\right\|_{s, \Omega_h\setminus\Omega}.
\end{equation}  
\end{lemma}
\begin{proof}
Consider the 2D case ($d = 2$). For simplicity, we can assume that $v$ is
$C^{\infty}$ regular and pass to $v \in H^s (\Omega_h)$ by density. By
Assumption \ref{asm0}, we can pass to the local coordinates $\xi_1, \xi_2$ on
every set $\mathcal{O}_k$ covering $\Gamma$ assuming that $\xi_1$ varies
between $0$ and $L$ and, for any $\xi_1$ fixed, $\xi_2$ varies on $\Omega_h
\setminus \Omega$ from 0 to some $b (\xi_1)$ with $0 \le b
(\xi_1) \le Ch$. \ We observe using the bounds on the mapping $(x_1, x_2)
\mapsto (\xi_1, \xi_2)$
\begin{align*}
  \|v\|_{0, (\Omega_h \setminus \Omega) \cap \mathcal{O}_k}^2 & \le C \int_0^L
  \int_0^{b (\xi_1)} v^2 (\xi_1, \xi_2) d \xi_2 d \xi_1\\
  & \qquad (\text{recall that } \frac{\partial^{\alpha} v}{\partial
  \xi_2^{\alpha}} (\xi_1, 0) = 0 \text{ for } \alpha = 0, \ldots, s \text{-1 and
  } b \le C_{} h)\\
  & = C \int_0^L \int_0^{b (\xi_1)} \left( \int_0^{\xi_2} \frac{(\xi_2 -
  t)^{s - 1}}{(s - 1) !} \frac{\partial^s v}{\partial \xi_2^s} (\xi_1, t) dt
  \right)^2 d \xi_2 d \xi_1\\
  & \le C \int_0^L h^{2 s}  \int_0^{b (\xi_1)} \left| \frac{\partial^s
  v}{\partial \xi_2^s} (\xi_1, t) \right|^2 dtd\xi_1\\
  & \le Ch^{2 s}  | v |^2_{s, (\Omega_h \setminus \Omega) \cap \mathcal{O}_k}.
\end{align*}
Summing over all neighbourhoods $\mathcal{O}_k$ gives (\ref{Dir:eq:4}). The
proof in the 3D case is the same up to the change of notations.
$\left.\right.$
\end{proof}

\subsection{Coercivity of the bilinear form $a_h$}

\begin{lemma}\label{lemma:coer}
  Under  Assumption \ref{asm2},  the bilinear form $a_h$ 
  is coercive on $V_h^{(k)}$ with respect to the norm 
  \[\interleave v_h\interleave_h:=\sqrt{| \phi_h v_h |_{1, \Omega_h}^2 + G_h
  (v_h, v_h)}\]
	i.e. $a_h(v_h,v_h)\ge c\interleave v_h\interleave_h^2$ for all $v_h\in V_h^{(k)}$ with $c>0$ depending only on the mesh regularity and on the constants in Assumption \ref{asm2}.
\end{lemma}

\begin{proof}
  Let $v_h \in V_h^{(k)}$ and $B_h$ be the strip between $\Gamma_h$ and $\partial \Omega_h$, \textit{i.e.} $B_h=\{\phi_h>0\}\cap \Omega_h$. Since $\phi_h v_h=0$ on $\Gamma_h$, 
	\begin{multline*} 
	\int_{\partial\Omega_h} \frac{\partial (\phi_h v_h)}{\partial n} \phi_h v_h 
	= \int_{\partial B_h}  \frac{\partial
 (\phi_h  v_h)}{\partial n} \phi_h v_h\\ = \sum_{T \in \mathcal{T}_h^{\Gamma}} \int_{
  \partial (B_h \cap T)}  \frac{\partial(\phi_h v_h)}{\partial n}\phi_h v_h
	- \sum_{T \in \mathcal{T}_h^{\Gamma}} \sum_{E \in \mathcal{F}_h^{cut}(T)} \int_{
  B_h\cap E}  \frac{\partial(\phi_h v_h)}{\partial n}\phi_h v_h,
	\end{multline*}
	where $\mathcal{T}_h^{\Gamma}$ is  defined in \eqref{eq:def ThGamma} and $\mathcal{F}_h^{cut}(T)$ regroups the facets of a mesh element $T$ cut by $\Gamma_h$. By divergence theorem,
	\begin{multline*}
\int_{\partial\Omega_h} \frac{\partial (\phi_h v_h)}{\partial n} \phi_h v_h 
   =\sum_{T \in \mathcal{T}_h^{\Gamma}} \int_{ B_h \cap T} |   \nabla(\phi_h v_h) |^2+\sum_{T \in
   \mathcal{T}_h^{\Gamma}}\int_{B_h\cap T}   \Delta(\phi_h v_h)\phi_h v_h\\
  - \sum_{E \in \mathcal{F}_h^{\Gamma}} \int_{E \cap B_h}\phi_h v_h
   \left[ \frac{\partial \phi_h v_h}{\partial n} \right] .
	\end{multline*} 
	Substituting this into the definition of $a_h$ yields
 \begin{multline}\label{ahvhvh}
    a_h (v_h, v_h) = \int_{\Omega_h} | \nabla(\phi_h v_h) |^2 - \sum_{T \in \mathcal{T}_h^{\Gamma}} \int_{\partial
   B_h \cap T} | \nabla
  (\phi_h  v_h) |^2-\displaystyle\sum_{T \in
   \mathcal{T}_h^{\Gamma}}\displaystyle\int_{B_h\cap T}   \Delta(\phi_h v_h)\phi_h v_h\\
   + \sum_{F \in \mathcal{F}_h^{\Gamma}} \int_{F \cap B_h} \phi_hv_h \left[
    \frac{\partial (\phi_h v_h)}{\partial n} \right]
    + \sigma h^2 \sum_{T \in\mathcal{T}_h^{\Gamma}} \int_T | \Delta(\phi v_h)|^2  
     +  \sigma h\sum_{E \in \mathcal{F}_h^{\Gamma}} \int_E \left[
    \frac{\partial (\phi_h v_h)}{\partial n} \right]^2.
  \end{multline}
Since   $B_h \subset \Omega_h^{\Gamma}$ (cf. \eqref{eq:def ThGamma}), 
applying Lemma \ref{LemDir:prop1} to $\phi_hv_h\in V_h^{(k+l)}$ gives
 \begin{multline*}
\sum_{T \in \mathcal{T}_h^{\Gamma}} \int_{\partial
   B_h \cap T} | \nabla  (\phi_h  v_h) |^2\le 
\alpha\int_{\Omega_h} | \nabla  (\phi_h  v_h) |^2 
+   \beta h\sum_{E \in \mathcal{F}_h^{\Gamma}} \int_E \left[
    \frac{\partial (\phi_h v_h)}{\partial n} \right]^2\\
    +\beta h^2 \sum_{T \in\mathcal{T}_h^{\Gamma}} \int_T | \Delta(\phi_h v_h)|^2  .
  \end{multline*}
Moreover, by Young inequality, \eqref{Dir:eq:3} and \eqref{Dir:prop2}, we obtain for any $\varepsilon>0$
 \[
\sum_{T \in
   \mathcal{T}_h^{\Gamma}}\displaystyle\int_{B_h\cap T}  \Delta(\phi_h v_h)\phi_h v_h
\le \frac{h^2}{2\varepsilon} \sum_{T \in\mathcal{T}_h^{\Gamma}} \int_T | \Delta(\phi_h v_h)|^2
+C\varepsilon  \int_{\Omega_h} |\nabla(\phi_h v_h)|^2
\]
and 
 \[
\sum_{F \in \mathcal{F}_h^{\Gamma}} \int_{F \cap B_h} \phi_hv_h \left[
    \frac{\partial (\phi_h v_h)}{\partial n} \right]
 \le \frac{h}{2\varepsilon} \sum_{E \in \mathcal{F}_h^{\Gamma}} \int_E \left[
    \frac{\partial (\phi_h v_h)}{\partial n} \right]^2
    +C\varepsilon    |\nabla(\phi v_h)|^2. 
 \]
Thus, putting the last 3 bounds into (\ref{ahvhvh}) we arrive at
  \begin{multline*}
    a (v_h, v_h) \ge \left( 1 - \alpha - C\varepsilon  \right) |\phi_h v_h
    |_{1, \Omega_h}^2 \\  
    + \left( \sigma - \beta - \frac{1}{2 \varepsilon}
    \right) h \sum_{E \in \mathcal{F}_h^{\Gamma}}\left\| \left[ \frac{\partial
   (\phi_h v_h)}{\partial n} \right] \right\|_{0, E}^2
    + \left( \sigma - \beta - \frac{1}{2 \varepsilon}
    \right) h^2 \sum_{T \in\mathcal{T}_h^{\Gamma}} \int_T | \Delta(\phi v_h)|^2   .
  \end{multline*}
This leads to the conclusion  taking $\varepsilon$ sufficiently small and $\sigma$ sufficiently big.
\end{proof}

\subsection{Proof of the $H^1$ error estimate in Theorem \ref{th:error}}

  Since $f\in H^{k}(\Omega)$, the solution $u$ of (\ref{eq:poisson}) belongs to $H^{k + 2}
  (\Omega)$ (see {\cite[p. 323]{evans}}) and can be extended by a function
  $\tilde{u}$ in $H^{k + 2} (\mathcal{O})$, cf. {\cite[p. 257]{evans}}, \ such
  that $\tilde{u} = u$ on $\Omega$ and
  \begin{equation}\label{ExtEst1}
	 \| \tilde{u} \|_{k + 2, \Omega_h} \le \| \tilde{u} \|_{k + 2,
     \mathcal{O}} \le C \| u \|_{k + 2, \Omega} \le C \|f\|_{k,
     \Omega} . 
	\end{equation}	
  Let $w = \tilde u / \phi$. By Lemma \ref{lemma:hardy},
  \begin{equation}\label{ExtEst2}
   | w |_{k + 1, \Omega_h} \le C \|u\|_{k + 2, \mathcal{O}} \le
     C \|f\|_{k, \Omega} . 
	\end{equation}
  Introduce the bilinear form $\bar{a}_h$, similar to $a_h$ as defined in
  (\ref{ah}) but with $\phi$ instead of $\phi_h$ multiplying the trial function:
  \begin{align*}
	  \bar{a}_h (w, v) &= \int_{\Omega_h} \nabla (\phi w) \cdot \nabla
     (\phi_h v) - \int_{\partial \Omega_h} \frac{\partial}{\partial n} 
     (\phi w) \phi_h v \\
		&+ \sigma h \sum_{E \in \mathcal{F}_h^{\Gamma}}
     \int \left[ \frac{\partial}{\partial n} (\phi w) \right]  \left[
     \frac{\partial}{\partial n} (\phi_h v) \right] + \sigma h^2  \sum_{T
     \in \mathcal{T}_h} \int_T \Delta (\phi w) \Delta (\phi_h v) .
	\end{align*}
  Since $\phi w = \tilde{u} \in H^2 (\Omega_h)$, an integration by parts
  yields
  \[
    \bar{a}_h (w, v_h) = \int_{\Omega_h} \tilde{f} \phi_h v_h - \sigma h^2 
    \sum_{T \in \mathcal{T}_h^\Gamma} \int_T \tilde{f} \Delta (\phi_h v_h), \quad
    \forall v_h \in V_h
  \]  
  with $\tilde{f} = - \Delta \tilde{u}$ on $\Omega_h$. Hence,
  \begin{equation}\label{GalOrt}
    a_h (w_h, v_h) - \bar{a}_h (w, v_h) = \int_{\Omega_h} (f - \tilde{f})
     \phi_h v_h - \sigma h^2  \sum_{T \in \mathcal{T}_h^\Gamma} \int_T (f -
     \tilde{f}) \Delta (\phi_h v_h) .
  \end{equation}
  Put $v_h = w_h - I_h w$. The last equality can be rewritten as
    \begin{align*}
     a_h (v_h, v_h) &= \bar{a}_h (w, v_h) - a_h (I_h w, v_h) \\
     &\quad + \int_{\Omega_h} (f - \tilde{f}) \phi_h v_h 
		   - \sigma h^2  \sum_{T \in \mathcal{T}_h^\Gamma}
        \int_T (f - \tilde{f}) \Delta (\phi_h v_h) \\
     &= \int_{\Omega_h} \nabla (\phi w - \phi_h I_hw ) \cdot \nabla
    (\phi_h v_h) - \int_{\partial \Omega_h} \frac{\partial}{\partial n} 
    (\phi_h w - \phi I_hw ) \phi_h v_h\\
    & \quad+ \sigma h \sum_{E \in \mathcal{F}_h^{\Gamma}} \int \left[
    \frac{\partial}{\partial n} (\phi w - \phi_h I_hw ) \right] 
    \left[ \frac{\partial}{\partial n} (\phi_h v_h) \right]\\
    &\quad + \sigma h^2 \sum_{T \in \mathcal{T}_h^{\Gamma}} \int_T \Delta (\phi w - \phi_h
    I_hw ) \Delta (\phi_h v_h)\\
    & \quad+ \int_{\Omega_h} (f - \tilde{f}) \phi_h v_h - \sigma h^2  \sum_{T \in
    \mathcal{T}_h} \int_T (f - \tilde{f}) \Delta (\phi_h v_h) .
  \end{align*}
    By Lemma \ref{lemma:coer}, Young inequality, and recalling $f = \tilde{f}$
  on $\Omega$, we now get
    \begin{align*}
    c \interleave v_h \interleave_h^2 &\le \frac{1}{2 \varepsilon}  |
    \phi w - \phi_h I_hw  |_{1, \Omega_h}^2  + \frac{h}{2 \varepsilon}  \left\|
    \frac{\partial}{\partial n} (\phi w - \phi_h I_hw ) \right\|_{0,
    \partial \Omega_h}^2 \\
    &+ \frac{\sigma^2 h}{2 \varepsilon}  \sum_{E \in \mathcal{F}_h^{\Gamma}}
    \left\|\left[ \frac{\partial}{\partial n} (\phi_h w - \phi I_hw )
    \right]\right\|_{0, E}^2
    +\frac{\sigma^2 h^2}{2 \varepsilon}  \sum_{T \in \mathcal{T}_h^{\Gamma}} \|
    \Delta (\phi_h w - \phi I_hw ) \|_{0, T}^2 \\
    &+ \frac{(1 +\sigma^2) h^2}{2 \varepsilon}  \|f - \tilde{f} \|_{0, \Omega_h \setminus
    \Omega}^2	\\
		&+  \frac{\varepsilon}{2}  \left( | \phi_h v_h |_{1, \Omega_h}^2 +\frac{1}{h} \| \phi_h v_h \|_{0, \partial \Omega_h}^2
		+ h \sum_{E \in\mathcal{F}_h^{\Gamma}} \left\|\left[ \frac{\partial}{\partial n} (\phi_h v_h)\right]
\right\|_{0, E}^2 \right.\\
    &\qquad \left. + 2 h^2  \sum_{T \in \mathcal{T}_h^{\Gamma}} \| \Delta
(\phi_h v_h) \|_{0, T}^2 + \frac{1}{h^2} \| \phi_h v_h \|_{0, \Omega_h
\setminus \Omega}^2 \right). 
  \end{align*}  
  We now show how to absorb the term with a coefficient $\varepsilon$ by the
  left-hand side. The first contribution $| \phi_h v_h |_{1, \Omega_h}$ and the sums over $\mathcal{F}_h^{\Gamma}$ and $\mathcal{T}_h^{\Gamma}$ are evidently controlled by $\interleave v_h \interleave_h$. Remark \ref{RemPrelim1} and Lemma \ref{Prelim2} give
  \[ 
	  \| \phi_h v_h \|_{0, \Omega_h\setminus \Omega} \le Ch | \phi_h v_h|_{1, \Omega_h}  \]
and  
  \[
    \| \phi_h v_h \|_{0, \partial \Omega_h}  \le C \sqrt{h}  | \phi_h v_h |_{1, \Omega_h}
  \]
so that these terms are also controlled by $\interleave v_h \interleave_h$. Taking $\varepsilon$ small enough, we conclude
    \begin{multline}\label{vhCea}
  \interleave v_h \interleave_h \le C \left( | \phi w - \phi_h
    I_hw  |_{1, \Omega_h}^2 + h \left\| \frac{\partial}{\partial n} (\phi
    w - \phi_h I_hw ) \right\|_{0, \partial \Omega_h}^2 \hspace*{3cm}\right.\\\left.\displaystyle+ h^2  \sum_{T \in
    \mathcal{T}_h^{\Gamma}} \| \Delta (\phi_h w - \phi I_hw ) \|_{0,
    T}^2 
    +  h \sum_{E \in \mathcal{F}_h^{\Gamma}} \left\|
    \frac{\partial}{\partial n} (\phi_h w - \phi I_hw ) \right\|_{0,
    E}^2 \right.\\\left.+ h^2 \|f - \tilde{f} \|_{0, \Omega_h \setminus \Omega}^2 \right) ^\frac{1}{2}.
 \end{multline}
      We now estimate each term in the right-hand side of (\ref{vhCea}). 
  By triangular inequality,
    \begin{align*}
    | \phi w - \phi_h I_hw  |_{1, \Omega_h} & \le | (\phi -
    \phi_h) w|_{1, \Omega_h} + | \phi_h (w - I_h w) |_{1, \Omega_h} \\
   &\le \| \nabla (\phi - \phi_h) \|_{L^{\infty} (\Omega_h)} \| w
     \|_{0, \Omega_h} + \| \phi - \phi_h \|_{L^{\infty} (\Omega_h)}
     |w|_{1, \Omega_h} \\
		&\quad + \| \nabla \phi_h  \|_{L^{\infty} (\Omega_h)} \| w
     - I_h w \|_{0, \Omega_h} + \| \phi_h \|_{L^{\infty} (\Omega_h)} |w -
     I_h w|_{1, \Omega_h} .
  \end{align*}
  We continue using the classical  interpolation bounds (see for instance {\cite{brenner2007mathematical})
  \begin{align*} 
	| \phi w - \phi_h I_hw  |_{1, \Omega_h} &\le Ch^k (| \phi
     |_{W^{k + 1, \infty} (\Omega_h)} \| w \|_{0, \Omega_h} + | \phi
     |_{W^{k, \infty} (\Omega_h)} |w|_{1, \Omega_h} \\
		& \quad + | \phi |_{W^{1,
     \infty} (\Omega_h)} |w|_{k, \Omega_h} + \| \phi \|_{L^{\infty}
     (\Omega_h)} |w|_{k + 1, \Omega_h}) \\
    & \le Ch^k \| \phi \|_{W^{k + 1, \infty} (\Omega_h)} \|w\|_{k + 1,\Omega_h} .
  \end{align*}
  Similarly, 
  \[ \left( \sum_{T \in \mathcal{T}_h} | \phi w - \phi_h I_hw  |_{2,
     T}^2 \right)^{\frac{1}{2}} \le Ch^{k - 1} \| \phi \|_{W^{k + 1,
     \infty} (\Omega_h)} \|w\|_{k + 1, \Omega_h} .\]
  Combining this with the trace inequality (\ref{eq:trace}), we conclude
   \begin{align*}
    \left\| \frac{\partial}{\partial n} (\phi_h w - \phi I_hw )
    \right\|_{0, \partial \Omega_h}^2 &+ \sum_{E \in \mathcal{F}_h^{\Gamma}}
    \left\| \left[ \frac{\partial}{\partial n} (\phi_h w - \phi I_hw )
    \right] \right\|_{0, E}^2\\
    & \le C \left( \frac{1}{h} \sum_{T \in \mathcal{T}_h^\Gamma} | \phi_h w
    - \phi I_hw  |_{1, T}^2 + h \sum_{T \in \mathcal{T}_h^\Gamma} |
    \phi_h w - \phi I_hw  |_{2, T}^2 \right)\\
    & \le Ch^{2 k - 1} \| \phi \|^2_{W^{k + 1, \infty} (\Omega_h)}
    \|w\|_{k + 1, \Omega_h}^2 .
  \end{align*}
  Finally, we get by Lemma \ref{Prelim15} applied to $f - \tilde{f}$ which vanishes on $\Omega$,
  \begin{equation}	\label{ftildef}
	\|f - \tilde{f} \|_{0, \Omega_h \setminus \Omega} \le Ch^{k - 1} 
     \|f - \tilde{f} \|_{k - 1, \Omega_h \setminus \Omega} \le Ch^{k -
     1}  (\|f\|_{k - 1, \Omega_h} + \| \tilde{u} \|_{k + 1, \Omega_h}) 
	\end{equation} 	
	since $\tilde{f}=-\Delta\tilde{u}$.
	
  Putting all these bounds into (\ref{vhCea}), we get
	\begin{equation}\label{triplevh}
	 | \phi_h (w_h - I_h w) |_{1, \Omega_h} \le \interleave v_h
     \interleave_h  
		 \le Ch^k ( \|w\|_{k + 1, \Omega_h} +\|f\|_{k - 1, \Omega_h} + \| \tilde{u} \|_{k + 1, \Omega_h}) . 
  \end{equation}
	We have absorbed $\| \phi \|_{W^{k + 1, \infty} (\Omega_h)}$ into the constant $C$ in the bound above. Indeed, the constants denoted by $C$ in this proof are allowed to depend on the constants from Assumption \ref{asm0}, which bound in particular $\| \phi \|_{W^{k + 1, \infty} (\Omega_h)}$. We shall follow the same convention on constants $C$ until the end of this proof. 
	
  By triangle inequality and interpolation bounds, 
  \begin{multline*} |u - \phi_h w_h |_{1, \Omega \cap \Omega_h} \le | \tilde{u} -
     \phi_h w_h |_{1, \Omega_h} \\ \le | (\phi - \phi_h) w|_{1,
     \Omega_h} + | \phi_h (w - I_h w) |_{1, \Omega_h} + | \phi_h (I_h w
     - w_h) |_{1, \Omega_h} \\
   \le Ch^k ( \|w\|_{k + 1, \Omega_h} +\|f\|_{k - 1,
     \Omega_h} + \| \tilde{u} \|_{k + 1, \Omega_h}) . 
	\end{multline*}
  We have thus proven (\ref{H1err}) taking into account the bounds (\ref{ExtEst1}) and (\ref{ExtEst2}).
	
\subsection{Proof of the $L^2$ error estimate in Theorem \ref{th:error}}
Let $z \in H^3 (\Omega)$ be solution to
\[ \left\{ \begin{array}{cl}
     - \Delta z = u - u_h & \text{in } \Omega,\\
     z = 0 & \text{on } \Gamma .
   \end{array} \right. \]
Extend it to $\Omega_h$ by $\tilde{z} \in H^3 (\Omega_h)$ using an extension
operator bounded in the $H^3$ norm. Set $ y  = \tilde{z} / \phi$. Then
\begin{equation}\label{L2y2}
  | y  |_{2, \Omega_h} \le C | \tilde{z} |_{3, \Omega_h} \le C \|u
- u_h \|_{1, \Omega}
\end{equation}  
thanks to Lemma \ref{lemma:hardy} and to the elliptic
regularity estimate. We also have 
\begin{equation}\label{L2y1}
 \|  y  \|_{1, \Omega_h} \le C \|\tilde{z} \|_{2, \Omega_h} \le C \|u - u_h \|_{0, \Omega}.
\end{equation}  

By Lemma 3.1 from \cite{lozinski2018}, we have for any $v\in H^1(\Omega_h^\Gamma)$
\begin{equation}
  \label{eq: zbound0} 
	\|v\|_{0, \Omega_h^{\Gamma}} \le C \left(
  \sqrt{h}  \| v \|_{0, \Gamma} + h | v |_{1,\Omega_h^{\Gamma}} \right) .
\end{equation}
This is valid since $\Omega_h^{\Gamma}$ is a band of thickness $\sim h$ around $\Gamma$. Note that the same estimate also holds for $\|v\|_{\Omega_h\setminus\Omega}$ (typically $\Omega_h\setminus\Omega\subset\Omega_h^\Gamma$, but even if it is not the case, $\Omega_h\setminus\Omega$ is still a band of thickness $\sim h$). In the case $v=\tilde{z}$, (\ref{eq: zbound0}) gives
\begin{equation}
  \label{eq: zbound1} \| \tilde{z} \|_{0, \Omega_h^{\Gamma}} 
  \le Ch |
  \tilde{z} |_{1, \Omega_h^{\Gamma}} \le Ch \|u - u_h \|_{0, \Omega}
\end{equation}
and, in the case $v=\nabla\tilde{z}$,
\begin{equation}
  \label{eq: zbound2} | \tilde{z} |_{1, \Omega_h^{\Gamma}} \le C \left(
  \sqrt{h}  \| \nabla \tilde{z} \|_{0, \Gamma} + h | \tilde{z} |_{2,
  \Omega_h^{\Gamma}} \right) \le C \sqrt{h}  \| \tilde{z} \|_{2, \Omega_h} \le
  C \sqrt{h}  \|u - u_h \|_{0, \Omega}.
\end{equation}

By integration by parts,
\begin{equation}
  \label{L2:1} \|u - u_h \|_{0, \Omega}^2 = \int_{\Omega_{}}  (u - u_h) (-
  \Delta z) = - \int_{\Gamma} (u - u_h)  \frac{\partial z}{\partial n} +
  \int_{\Omega_{}} \nabla (u - u_h) \cdot \nabla z.
\end{equation}
To treat the first term in (\ref{L2:1}), we remark first
\[ \int_{\Gamma} (u - u_h)  \frac{\partial z}{\partial n} \le \|u - u_h
   \|_{0, \Gamma} \left\| \frac{\partial z}{\partial n} \right\|_{0, \Gamma}
   \le C \|u - u_h \|_{0, \Gamma} \| u - u_h \|_{0, \Omega} . \]
Furthermore, since the distance between $\Gamma$ and $\Gamma_h$ is of order at
least $h^{k + 1}$, we have 
\begin{align*} \|u - u_h \|_{0, \Gamma} &\le C (\| \tilde{u} - u_h \|_{0, \Gamma_h} +
   h^{(k + 1) / 2} | \tilde{u} - u_h |_{1, \Omega_h}) \\
 &\qquad (\text{recalling } \tilde{u} = \phi w \text{ and } \phi_h = u_h = 0 \text{ on }
   \Gamma_h) \\
 & = C (\| (\phi - \phi_h) w\|_{0, \Gamma_h} + h^{(k + 1) / 2} | \tilde{u} -
   u_h |_{1, \Omega_h}) \\
 & \le C (h^{k + 1} \|w\|_{0, \Gamma_h} + h^{(k + 1) / 2 + k} \| f
   \|_{k, \Omega_h}). 
\end{align*}
We have used here the already proven bound on $| \tilde{u} - u_h |_{1,
\Omega_h}$ and the interpolation error bound for $\phi - \phi_h .$ We have
thus thanks to Lemma \ref{lemma:hardy},
\begin{multline*} \|u - u_h \|_{0, \Gamma} \le Ch^{k + 1} (\|w\|_{1, \Omega_h} + \| f \|_{k, \Omega_h}) \le Ch^{k + 1} (\| \tilde{z} \|_{2, \Omega_h} + \|
   f \|_{k, \Omega_h})\\ \le Ch^{k + 1} (\|u - u_h \|_{0, \Omega_{}} + \|
   f \|_{k, \Omega_h}) .
   \end{multline*}
Hence,
\begin{equation}\label{L2 first term} 
 \int_{\Gamma} (u - u_h)  \frac{\partial z}{\partial
   n} \le Ch^{k + 1} (\|u - u_h \|_{0, \Omega_{}}^2 + \| f \|_{k,
   \Omega_h} \|u - u_h \|_{0, \Omega_{}}) . 
\end{equation}

The second term in (\ref{L2:1}) is treated by Galerkin orthogonality
(\ref{GalOrt}): for any $ y _h \in V_h^{(k)}$ 
\begin{equation}\label{L2 second term} 
  \int_{\Omega} \nabla (u - u_h) \cdot \nabla z =
   \underbrace{\int_{\Omega_h} \nabla (\phi w - \phi_h w_h) \cdot \nabla (\phi
    y  - \phi_h  y _h)}_I 
\end{equation} 
\[ -\underbrace{\int_{\Omega_h \setminus \Omega} \nabla (\phi w - \phi_h w_h)
   \cdot \nabla (\phi  y )}_{II} \]
\[ +\underbrace{\int_{\partial \Omega_h} \frac{\partial}{\partial n}  (\phi w
   - \phi_h w_h)  (\phi_h  y _h)}_{III} \]
\[ -\underbrace{ \sigma h \sum_{E \in \mathcal{F}_h^{\Gamma}} \int_E \left[
   \frac{\partial}{\partial n} (\phi w - \phi_h w_h) \right]  \left[
   \frac{\partial}{\partial n} (\phi_h  y _h) \right]
   -
   \sigma h^2  \sum_{T \in \Th^{\Gamma}} \int_T \Delta (\phi w - \phi_h w_h)
   \Delta (\phi_h  y _h)}_{IV} \]
\[ +\underbrace{\int_{\Omega_h} (f - \tilde{f}) \phi_h  y _h - \sigma h^2 
   \sum_{T \in \mathcal{T}_h^{\Gamma}} \int_T (f - \tilde{f}) \Delta (\phi_h
    y _h)}_{V}. \]
We now estimate term by term the right-hand side of the above inequality taking $y_h=\tilde{I}_hy$ with $\tilde{I}_h$ the Cl\' ement interpolation operator on $\Th$. We shall skip some tedious technical details as they are similar to thouse in the proof of the $H^1$ error estimate above. We recall that we do not track explicitly the dependence of constants on the norms of $\phi$.

\noindent{\bf Term I}: by Cauchy-Schwartz, the already proven bound on 
$| \tilde{u} -u_h |_{1, \Omega_h}$, and $(\ref{L2y2})$
\begin{multline*} | I | \le C| \tilde{u} -u_h |_{1, \Omega_h} | \phi  y  - \phi_h  y _h
   |_{1, \Omega_h} \le Ch^{k + 1} \|f\|_{k, \Omega_h} \|  y  \|_{2,
   \Omega_h} \\\le Ch^{k + 1} \|f\|_{k, \Omega_h} \| \tilde{u} - u_h
   \|_{1, \Omega} .
   \end{multline*}

\noindent{\bf Term II}: using (\ref{eq: zbound1}) for $\tilde{z} = \phi  y $,
\[ | II | \le | \tilde{u} - u_h |_{1, \Omega_h} \text{} | \tilde{z}
   |_{1, \Omega_h \setminus \Omega} \text{} \le Ch^{k + 1 / 2} \|f\|_{k,
   \Omega_h}  \|u - u_h \|_{0, \Omega} . \]

\noindent{\bf Term III}: applying the trace inequality on the mesh elements adjacent to
$\partial \Omega_h$ yields
\[ | III | \le \left( \sum_{T \in \mathcal{T}_h^{\Gamma}} \left\{
   \frac{1}{h} | \tilde{u} - u_h |_{1, T}^2 + \sum_{T \in \mathcal{T}_h^{\Gamma}}h | \tilde{u} - u_h |_{2, T}^2
   \right\} \right)^{1 / 2} \| \phi_h  y _h \|_{0, \partial \Omega_h} .\]
The term with the sum over $T \in \mathcal{T}_h^{\Gamma}$ can be further
bounded using the triangle inequality, interpolation estimates, and the bound \eqref{triplevh}
on $v_h = \phi_h (w_h - I_h w)$ as
\begin{multline*} (\cdots)^{1 / 2} \le \left( \frac{1}{h} | \tilde{u} - \phi_h I_h w
   |_{1, \Omega_h^{\Gamma}}^2 + h | \tilde{u} - \phi_h I_h w |^2_{^{} 2, T}
   \right)^{1 / 2} + \frac{1}{\sqrt{h}} \interleave v_h \interleave_h \\
 \le Ch^{k - 1 / 2} \| f \|_{k, \Omega_h}. 
\end{multline*}
Moreover, since the distance between $\Gamma_h$ and $\partial \Omega_h$ is of
order $h$, we have 
\[\| \phi_h \|_{L^{\infty} (\partial \Omega_h)} \le
Ch\| \nabla\phi_h \|_{L^{\infty} (\partial \Omega_h)} \le
Ch\]
 and, by (\ref{L2y1}),
\[ \| \phi_h  y _h \|_{0, \partial \Omega_h} \le Ch \|  y _{} \|_{1, \Omega_h} \le Ch \|
   {u} - u_h \|_{0, \Omega} \]
so that
\[ | III | \le Ch^{k + 1 / 2} \| f \|_{k, \Omega_h} \| u - u_h \|_{0,
   \Omega} .\]

\noindent{\bf Term IV}: applying the trace inequality on the mesh elements adjacent to
$\partial \Omega_h$ yields
\[ | IV | \le (Ch^k \| f \|_{k, \Omega_h} + \interleave v_h
   \interleave_h) G_h ( y _h,  y _h)^{1 / 2} \le Ch^k \| f \|_{k,
   \Omega_h} G_h ( y _h,  y _h)^{1 / 2} \]
and by \eqref{eq: zbound2}
\begin{multline}\label{L2:10}
 G_h ( y _h,  y _h)^{1 / 2} \le \frac{C}{h} \| \phi_h  y _h
   \|_{0, \Omega_h^{\Gamma}} \le C \|  y _h \|_{0, \Omega_h^{\Gamma}}
\\
 \le C \|  y _{} \|_{0, \Omega_h^{\Gamma}} \le C |
   \tilde{z} |_{1, \Omega_h^{\Gamma}} \le C \sqrt{h} \|u - u_h \|_{0,
   \Omega} .
\end{multline}	
Hence,
\[ | IV | \le Ch^{k + 1 / 2} \| f \|_{k, \Omega_h} \|u - u_h \|_{0,
   \Omega}. \]

\noindent{\bf Term V}: by an inverse inequality and \eqref{ftildef}
\[ | V | \le \| f - \tilde{f} \|_{0, \Omega_h \setminus \Omega}  \|
   \phi_h  y _h \|_{0, \Omega_h \setminus \Omega} 
	\le Ch^{k-1}  \|f\|_{k,\Omega_h}  \| \phi_h  y _h \|_{0, \Omega_h \setminus \Omega}.\]
As we have already proved in (\ref{L2:10})
\[ \| \phi_h  y _h \|_{0, \Omega_h^{\Gamma}} \le Ch^{3 / 2}  \|u -
   u_h \|_{0, \Omega} \]
we conclude
\[ | V | \le Ch^{k + 1 / 2} \| f \|_{k, \Omega_h} \|u - u_h \|_{0,
   \Omega} .\]

Combining the bounds for the terms I--V in (\ref{L2 second term})  with (\ref{L2 first term})
and putting all this into (\ref{L2:1}), we obtain by Young inequality
\begin{multline*}
 \|u - u_h \|_{0, \Omega}^2 \le C (h^{k + 1} \|u - u_h \|_{0,
   \Omega_{}}^2 + h^{k + 1 / 2} \| f \|_{k, \Omega_h} \|u - u_h \|_{0,
   \Omega_{}}\\ + h^{k + 1} \| f \|_{k, \Omega_h} \|u - u_h \|_{1, \Omega_{}})\\
 \le Ch^{k + 1} \|u - u_h \|_{0, \Omega_{}}^2 + \frac{C}{\varepsilon}
   h^{2 k + 1} \|f\|_{k, \Omega_h}^2 + \varepsilon \|u - u_h \|_{0, \Omega}^2 +
   \varepsilon h \|u - u_h \|_{1, \Omega}^2 .
   \end{multline*}
By the already established estimate for $|u - u_h |_{1, \Omega}$,
\[ \| u - u_h \|_{0, \Omega}^2 \le C \left( \frac{1}{\varepsilon}
   + \varepsilon \right) h^{2 k + 1} \|f\|_{k, \Omega_h}^2 + (Ch^{k + 1} +
   \varepsilon + \varepsilon h)  \|u - u_h \|_{0, \Omega}^2 \]
which  proves (\ref{L2err}) taking sufficiently small $\varepsilon$ and
supposing $h$ small enough.

\section{Conditioning of the system matrix}

We are now going to prove that the condition number of the finite element
matrix associated to the bilinear form $a_h$ of $\phi$-FEM does
not suffer from the introduction of the multiplication by $\phi_h$: it is of order $1 / h^2$
on a quasi-uniform mesh of step $h$, similar to the standard FEM on a fitted mesh.

\begin{theorem}
  [Conditioning]\label{prop:cond gp} Under Assumptions \ref{asm0} and \ref{asm2} and recalling that the mesh $\Th$ is supposed to be quasi-uniform, the condition number
  ${\kappa} (\mathbf{A}) := \| \mathbf{A} \|_2 \|
  \mathbf{A}^{- 1} \|_2$ of the matrix $\mathbf{A}$ associated to the
  bilinear form $a_h$ on $V_h^{(k)}$, as in (\ref{ah}), satisfies
  \[ \kappa (\mathbf{A}) \le Ch^{- 2} . \]
  Here, $\| \cdot \|_2$ stands for the matrix norm associated to the vector
  2-norm $| \cdot |_2$.
\end{theorem}

Before proving Theorem \ref{prop:cond gp}, we introduce some
auxiliary results:

\begin{lemma}\label{lemma:coer ah} Under the assumptions of
  Theorem \ref{prop:cond gp}, it holds for all $w_h \in V_h^{(k)}$
  \[ a_h (w_h, w_h) \ge C \|w_h \|_{0, \Omega_h}^2 . \]
\end{lemma}

\begin{proof}
  By Lemma \ref{lemma:coer}, it holds for each $w_h \in V_h^{(k)}$
  \[ a_h (w_h, w_h) \ge c \interleave w_h \interleave_h^2 \ge c |
     \phi_h w_h |^2_{1, \Omega_h} . \]
	We now denote $u_h=\phi_hw_h$ and apply Lemma \ref{lemma:hardy} with $k=0$ and $\phi_h$ instead of $\phi$ to $w_h=u_h/\phi_h$: 
	\begin{equation}\label{hardyH}
	\|w_h \|_{0, \Omega_h}\le C\|u_h \|_{1, \Omega_h}.
	\end{equation}
	This is justified by a possible relaxation of the hypotheses of Lemma \ref{lemma:hardy} as outlined in Remark \ref{RemAsm0}. The constant in (\ref{hardyH}) will depend on $\|\phi_h\|_{W^{1,\infty}}(\Omega_h)$ which is bounded uniformly in $h$. Moreover, the local coordinates around $\Gamma$ evoked in Assumption \ref{asm0} can be reused to build the same around $\Gamma_h$.  
	
	Applying Poincar\'e inequality on the domain $\Omega_h^{in}:=\{\phi_h<0\}$ yields, as $u_h=0$ on $\Gamma_h=\partial\Omega_h^{in}$,
	$$
	  \|u_h \|_{0, \Omega_h^{in}}\le C|u_h |_{1, \Omega_h^{in}}
	$$
	with a constant that depends only on the diameter of $\Omega_h^{in}$ and can be thus assumed $h$-independent. 
	Moreover, invoking Lemma \ref{Prelim1} and observing $\Omega_h\setminus\Omega_h^{in}\subset\Omega_h^{\Gamma}$ we conclude that
	\begin{equation}\label{poinH}
	\|u_h \|_{0, \Omega_h}\le C|u_h|_{1, \Omega_h}.
	\end{equation}
	Combining this with (\ref{hardyH}) we finish the proof as follows: 
  \[ a_h (w_h, w_h) \ge c | u_h |^2_{1, \Omega_h}\ge C \| u_h \|^2_{1, \Omega_h} \ge C
     \|w_h \|^2_{0, \Omega_h} . \]
\end{proof}

\begin{lemma}
  \label{lemma:cont ah} Under the assumptions of
  Theorem \ref{prop:cond gp}, it holds for all $u_h, w_h \in V_h^{(k)}$
  \[ a_h (w_h, v_h) \le \frac{C}{h^2} \|w_h \|_{0, \Omega} \|v_h \|_{0,
     \Omega} . \]
\end{lemma}

\begin{proof}
  It is sufficient to prove this statement for the case $w_h = v_h .$
  Let\quad$w_h \in V_h^{(k)}$. By definition of $a_h$ and Lemma \ref{Prelim2},
  \[ a_h (w_h, w_h) \le C | \phi_h w_h |^2_{1, \Omega_h} + C \sqrt{h} 
     \left\| \frac{\partial (\phi_h w_h)}{\partial n} \right\|_{0, \partial
     \Omega_h} | \phi_h w_h |_{1, \Omega_h} + Ch^2  \sum_{T \in
     \mathcal{T}_h^{\Gamma}} | \phi_h w_h |^2_{2, T} . \]
  Using the inverse inequalities on $V_h^{(k + l)}$ 
  \[ \left\| \frac{\partial (\phi_h w_h)}{\partial n} \right\|_{0, \partial
     \Omega_h} \le \frac{C}{\sqrt{h}}  \| \phi_h w_h \|_{0, \Omega_h},
     \quad | \phi_h w_h |_{1, \Omega_h} \le \frac{C}{h}  \| \phi_h w_h
     \|_{0, \Omega_h}, \] 
	and $| \phi_h w_h |_{2, T} \le\frac{C}{h^2}  \| \phi_h w_h \|_{0, T}$
	yields
		\[ a_h (w_h, w_h) \le C \| \phi_h w_h \|^2_{0, \Omega_h}\le C \| w_h \|^2_{0, \Omega_h}\]
  since $\phi_h$ is bounded uniformly in $h$.
\end{proof}

\begin{proof}[Proof of Theorem \ref{prop:cond gp}]
  Denote the dimension of $V_h^{(k)}$ by $N$ and let us associate any $v_h \in
  V_h^{(k)}$ with the vector $\mathbf{v} \in \mathbb{R}^N$ contaning the
  expansion coefficients of $v_h$ in the standard finite element basis.
  Recalling that the mesh is quasi-uniform and using the equivalence of norms on the reference element,
  we can easily prove that
  \begin{equation}
    \label{eq:vh v} C_1 h^{d / 2} |\mathbf{v}|_2 \le \|v_h \|_{0,
    \Omega_h} \le C_2 h^{d / 2} |\mathbf{v}|_2 .
  \end{equation}
  Inequality (\ref{eq:vh v}) with Lemma \ref{lemma:cont ah} imply
$$
\|\mathbf{A}\|_2=\sup_{\mathbf{v}\in\mathbb{R}^{ N}}\dfrac{(\mathbf{A}\mathbf{v},\mathbf{v})}{|\mathbf{v}|_2^2}=\sup_{\mathbf{v}\in\mathbb{R}^{ N}}\dfrac{a(v_h,v_h)}{|\mathbf{v}|_2^2}
\le Ch^d\sup_{v_h\in V_h}\dfrac{a(v_h,v_h)}{ \|v_h\|_0^2}
\le Ch^{d-2}.
$$
  Similarly, (\ref{eq:vh v}) with Lemma \ref{lemma:coer ah} imply
  \[ \| \mathbf{A}^{- 1} \|_2 = \sup_{\mathbf{v} \in \mathbb{R}^N} \frac{|
     \mathbf{v} |_2^2}{(\mathbf{{Av}}, \mathbf{v})} =
     \sup_{\mathbf{v} \in \mathbb{R}^N} \frac{| \mathbf{v} |_2^2}{a (v_h,
     v_h)} \le \frac{C}{h^d} \sup_{v_h \in V_h} \frac{\|v_h \|_0^2}{a
     (v_h, v_h)} \le \frac{C}{h^d} . \]
  These estimates lead to the desired result.
\end{proof}

\section{Numerical results} 
We have implemented $\phi$-FEM in FEniCS Project \cite{LoggMardal2012} and report here some results using uniform Cartesian meshes on a rectangle $\mathcal{O}$ as the backgound mesh $\Th^\mathcal{O}$. 

\subsection*{1$^{\mbox{st}}$ test case}
Let $\Omega$ be the circle of radius $\sqrt{2}/4$ centered at the point $(0.5,0.5)$ and the surrounding domain $\mathcal{O}=(0,1)^2$. The level-set function $\phi$ giving this domain $\Omega$ is taken as
\begin{equation}\label{eq:test1}
\phi(x,y)= 1/8-(x-1/2)^2-(y-1/2)^2.
\end{equation}
We use $\phi$-FEM to solve numerically Poisson-Dirichlet problem \eqref{eq:poisson} with the exact solution given by 
\begin{equation}\label{eq:test1e}
u(x,y)=\phi(x,y)\times\exp(x)\times\sin(2\pi y).
\end{equation}

The results with $P_1$ finite elements are reported in Fig. \ref{fig:simu}. We give there the evolution of the errors in $L^2$ and $H^1$ norms under the mesh refinement for $\phi$-FEM with stabilization parameter $\sigma=20$ and for $\phi$-FEM without stabilization, $\sigma=0$.  The numerical results confirm the theoretically predicted optimal convergence orders (in fact, the convergence order in the $L^2$ norm is 2 and is thus better than in theory). We also observe that the ghost stabilization is indeed crucial to ensure the convergence of the method. The level-set $\phi$ is approximated here by a $P_1$ finite element function $\phi_h$, \textit{i.e.} we take $l=k$ in (\ref{eq:phi_h}). Note that the choice $l=2$ is also possible and would result in $\phi_h$ reproducing $\phi$ exactly. In practice, it produces an approximation $u_h$ of nearly the same accuracy as those with $l=1$. We choose thus not to report these results here.  

The condition number of the matrix produced by $\phi$-FEM is numerically investigated at Fig. \ref{fig:simu cond}. In accordance with Theorem \ref{prop:cond gp}, the condition number is of order $1/h^2$ at worst. We observe that the ghost stabilization ($\sigma=20$) is necessary to obtain this nice conditioning: the condition numbers produced by the naive method with $\sigma=0$ become much higher as $h\to 0$. The influence of the stabilization parameter $\sigma$ on the accuracy of $\phi$-FEM is investigated  at Fig. \ref{fig:sigma}. We observe that the accuracy of the method is only slightly affected by the value of $\sigma$ provided it is not taken too small: $\sigma$ in the range $[0.1,20]$ produce very similar errors, especially when measured in the $H^1$ semi-norm.

We finally describe the results obtained with higher order $P_k$ finite elements, $k=2,3$. The errors are reported in Fig. \ref{fig:simu2}. The optimal convergence orders under the mesh refinement are again observed (with the order $(k+1)$ in the $L^2$ norm, which is thus better than in theory). The influence of the stabilization parameter $\sigma$ on the accuracy of $\phi$-FEM with $P_2$ finite elements is investigated  at Fig. \ref{fig:sigma2}. We observe that the method works fine and is robust with respect to the value of $\sigma$ at least in the range $[0.1,20]$ (the same as for the $P_1$ elements).

\subsection*{2$^{\mbox{nd}}$ test case} 
We now choose domain $\Omega$ given by the level-set
\begin{equation}\label{eq:test2}
\phi(x,y)= -(y-\pi x-\pi)\times(y+x/pi-\pi)\times(y-\pi x+\pi)\times(y+x/pi+\pi).
\end{equation}
It is thus the rectangle with corners   $ \left(\frac{2\pi^2}{\pi^2+1},\frac{\pi^3-\pi}{\pi^2+1}\right)$,
                  $\left(0, \pi\right)$,
									$ \left(-\frac{2\pi^2}{\pi^2+1},-\frac{\pi^3-\pi}{\pi^2+1}\right)$,
                  $\left(0, -\pi\right)$.
We use $\phi$-FEM to solve numerically Poisson-Dirichlet problem \eqref{eq:poisson} in $\Omega$ with the right-hand side given by 
\begin{equation}\label{eq:test2e}
f(x,y)=1.
\end{equation}
This test case is not consistent with Assumption \ref{asm0}. We want here to test $\phi$-FEM outside of the setting where it is theoretically justified.

The results with $P_1$ and $P_2$ finite elements are reported in Fig. \ref{fig:test2 simu1}. Notwithstanding the lack of theoretical justification, we observe the optimal convergence in the case $k=1$ and somewhat close to optimal convergence in the case $k=2$. Note that $\phi_h$ is approximated in both cases with $P_k$ finite elements, \textit{i.e.} $l=k$ in (\ref{eq:phi_h}). We do not have the exact solution in this test case. We compare thus the $\phi$-FEM solution $u_h$ against a reference solution given by standard FEM on a sufficiently fine mesh fitted to the rectangle $\Omega$.


\begin{figure}
\begin{center}
\begin{tikzpicture}[thick,scale=0.8, every node/.style={scale=1.0}]
\begin{loglogaxis}[xlabel=$h$,xmin=1e-3,xmax=0.3 ,ymin=1e-6
,legend pos=south east
,legend style={ font=\large}
, legend columns=1]
 \addplot[color=blue,mark=triangle*] coordinates { 
(0.141421356237,0.873511794417)
(0.0707106781187,0.236170953338)
(0.0353553390593,0.0461366869498)
(0.0176776695297,0.00723294973919)
(0.00883883476483,0.00118910559933)
(0.00441941738242,0.000216240408077)
(0.00220970869121,4.35840617933e-05)

 };
 \addplot[color=red,mark=*] coordinates { 
(0.141421356237,0.970905613109)
(0.0707106781187,0.2956050376)
(0.0353553390593,0.0887728723381)
(0.0176776695297,0.0320169580093)
(0.00883883476483,0.0140678444911)
(0.00441941738242,0.0067306206001)
(0.00220970869121,0.00329791584965)
 };
\logLogSlopeTriangle{0.33}{0.2}{0.5}{1}{red};
\logLogSlopeTriangle{0.53}{0.2}{0.35}{2}{blue};
 \legend{$\|u-u_h\|_{0,\Omega_h}/\|u\|_{0,\Omega_h}$,$|u-u_h|_{1,\Omega_h}/|u|_{1,\Omega_h}$}
\end{loglogaxis}
\end{tikzpicture}
\begin{tikzpicture}[thick,scale=0.8, every node/.style={scale=1.0}]
\begin{loglogaxis}[xlabel=$h$,xmin=1e-3,xmax=0.3 ,ymin=1e-6
,legend pos=south east
,legend style={ font=\large}
, legend columns=1]
 \addplot[color=blue,mark=triangle*] coordinates { 
(0.141421356237,0.154081014305)
(0.0707106781187,0.0430719640976)
(0.0353553390593,0.0106913662114)
(0.0176776695297,0.00312334015409)
(0.00883883476483,0.442265423669)
(0.00441941738242,0.000219709581083)
(0.00220970869121,0.00324894740378)
 };
 \addplot[color=red,mark=*] coordinates { 
(0.141421356237,0.245635558838)
(0.0707106781187,0.116604327112)
(0.0353553390593,0.053712029551)
(0.0176776695297,0.0363051641839)
(0.00883883476483,9.60586410614)
(0.00441941738242,0.00779795419309)
(0.00220970869121,0.269612583667)
 };
\logLogSlopeTriangle{0.83}{0.2}{0.75}{1}{red};
\logLogSlopeTriangle{0.83}{0.2}{0.35}{2}{blue};
 \legend{$\|u-u_h\|_{0,\Omega_h}/\|u\|_{0,\Omega_h}$,$|u-u_h|_{1,\Omega_h}/|u|_{1,\Omega_h}$}
\end{loglogaxis}
\end{tikzpicture}
\caption{Relative errors of $\phi$-FEM for the test case \eqref{eq:test1}--\eqref{eq:test1e} and $k=1$. Left: $\phi$-FEM  with ghost penalty $\sigma=20$; Right: $\phi$-FEM without ghost penalty ($\sigma=0$).  }
\label{fig:simu}
\end{center}\end{figure}
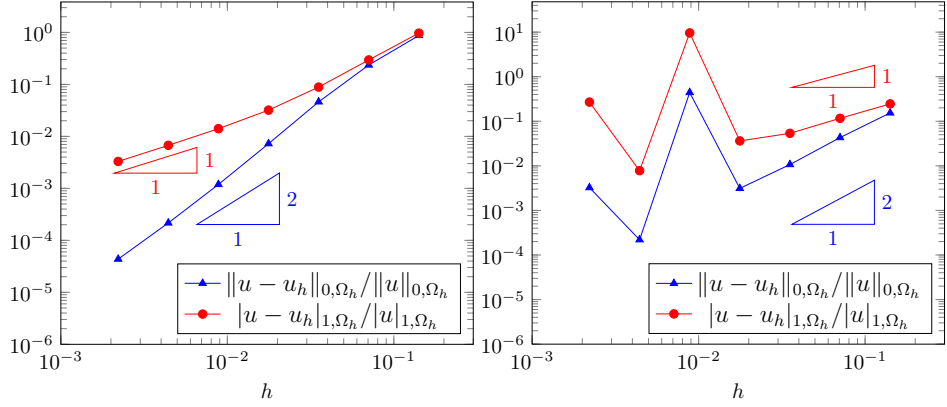

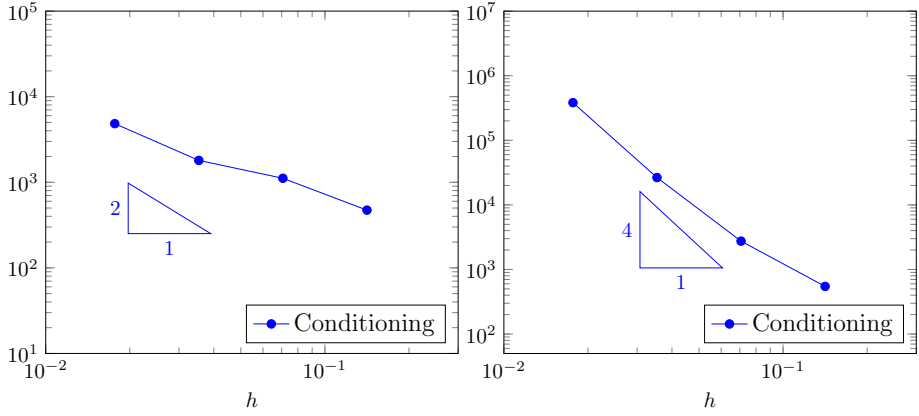
\begin{figure}
\begin{center}
\begin{tikzpicture}[thick,scale=0.8, every node/.style={scale=1.0}]
\begin{loglogaxis}[xlabel=$h$,xmin=1e-2,xmax=0.3 ,ymin=1e1,ymax=1e5
,legend pos=south east
,legend style={ font=\large}
, legend columns=1]
 \addplot[color=blue,mark=*] coordinates { 
(0.141421356237,472.653655919)
(0.0707106781187,1113.28190247)
(0.0353553390593,1801.3852442)
(0.0176776695297,4840.54481748)
 };
 \logLogSlopeTriangleinv{0.4}{0.2}{0.35}{2}{blue};
 \legend{Conditioning}
\end{loglogaxis}
\end{tikzpicture}
\begin{tikzpicture}[thick,scale=0.8, every node/.style={scale=1.0}]
\begin{loglogaxis}[xlabel=$h$,xmin=1e-2,xmax=0.3,ymin=0.5e2,ymax=1e7
,legend pos=south east
,legend style={ font=\large}
, legend columns=1]
 \addplot[color=blue,mark=*] coordinates { 
(0.141421356237,546.933790389)
(0.0707106781187,2748.32146661)
(0.0353553390593,26486.5110206)
(0.0176776695297,382886.824485)
 };
  \logLogSlopeTriangleinv{0.53}{0.2}{0.25}{4}{blue};
 \legend{Conditioning}
\end{loglogaxis}
\end{tikzpicture}
\caption{Condition numbers  for $\phi$-FEM in the test case \eqref{eq:test1} and $k=1$. Left: $\phi$-FEM  with ghost penalty $\sigma=20$; Right: $\phi$-FEM without ghost penalty ($\sigma=0$).  }
\label{fig:simu cond}
\end{center}\end{figure}

\begin{figure}
\begin{center}
\begin{tikzpicture}[thick,scale=0.8, every node/.style={scale=1.0}]
\begin{loglogaxis}[xlabel=$\sigma$,xmin=1e-5,xmax=2000 ,ymin=1e-5, ymax=1e1
,legend pos=north east
,legend style={ font=\large}
, legend columns=1]
 \addplot[color=green,mark=square*] coordinates { 
(100.0,0.013447942325)
(10.0,0.00269135544713)
(1.0,0.00146169487186)
(0.1,0.00121107405959)
(0.01,0.0019973652263)
(0.001,0.00211295791599)
(0.0001,0.00190737628392)
 };
 \addplot[color=blue,mark=triangle*] coordinates { 
(100.0,0.00197761952858)
(10.0,0.00051048114672)
(1.0,0.000358653370425)
(0.1,0.000314814356398)
(0.01,0.000466175191137)
(0.001,0.000952085523581)
(0.0001,0.00390443456128)
 };
 \addplot[color=red,mark=*] coordinates { 
(100.0,0.000290114436338)
(10.0,0.000105476737907)
(1.0,8.57831173449e-05)
(0.1,7.54217207653e-05)
(0.01,0.0125902639676)
(0.001,0.0135727850608)
(0.0001,0.062065603825)
 };
 \addplot[color=black,mark=diamond*] coordinates { 
(100.0,4.65839555038e-05)
(10.0,2.42114373731e-05)
(1.0,2.14485656293e-05)
(0.1,1.91431202155e-05)
(0.01,3.57675223357e-05)
(0.001,3.40902138993e-05)
(0.0001,3.03332551839e-05)
 };
 \legend{
  $h=1.41\times 10^{-2}$, $h=7.07\times 10^{-3}$, $h=3.54\times 10^{-3}$,
 $h=1.77\times 10^{-3}$}
\end{loglogaxis}
\end{tikzpicture}
\begin{tikzpicture}[thick,scale=0.8, every node/.style={scale=1.0}]
\begin{loglogaxis}[xlabel=$\sigma$,xmin=1e-5,xmax=2000,ymax=20
,legend pos=north east
,legend style={ font=\large}
, legend columns=1]
 \addplot[color=green,mark=square*] coordinates { 
(100.0,0.0370713574711)
(10.0,0.0228380664052)
(1.0,0.0210757847008)
(0.1,0.0208973567599)
(0.01,0.0288652213275)
(0.001,0.0236348653015)
(0.0001,0.0224867382674)
 };
 \addplot[color=blue,mark=triangle*] coordinates { 
(100.0,0.0128066016083)
(10.0,0.0107956873544)
(1.0,0.0104354354897)
(0.1,0.0105307893322)
(0.01,0.0112837460096)
(0.001,0.0270568303435)
(0.0001,0.126013336227)
 };
 \addplot[color=red,mark=*] coordinates { 
(100.0,0.00565376768203)
(10.0,0.0052790867387)
(1.0,0.00518956181338)
(0.1,0.00518023572183)
(0.01,0.52614985633)
(0.001,0.579340540734)
(0.0001,4.45062696478)
 };
  \addplot[color=black,mark=diamond*] coordinates { 
(100.0,0.00268851437048)
(10.0,0.00261098108991)
(1.0,0.00258957604048)
(0.1,0.00258719448302)
(0.01,0.00263260408141)
(0.001,0.00321895923917)
(0.0001,0.00331308428098)
 };
 \legend{
  $h=1.41\times 10^{-2}$, $h=7.07\times 10^{-3}$, $h=3.54\times 10^{-3}$,
 $h=1.77\times 10^{-3}$}
\end{loglogaxis}
\end{tikzpicture}
\caption{Influence of the ghost penalty parameter $\sigma$ on the relative errors for $\phi$-FEM in the test case \eqref{eq:test1}--\eqref{eq:test1e} and $k=1$. Left: $\|u-u_h\|_{0,\Omega_h}/\|u\|_{0,\Omega_h}$; Right: $|u-u_h|_{1,\Omega_h}/|u|_{1,\Omega_h}$.  }
\label{fig:sigma}
\end{center}\end{figure}

\begin{figure}
\begin{center}
\begin{tikzpicture}[thick,scale=0.8, every node/.style={scale=1.0}]
\begin{loglogaxis}[xlabel=$h$,xmin=1e-3,xmax=0.3 ,ymin=1e-11
,legend pos=south east
,legend style={ font=\large}
, legend columns=1]
 \addplot[color=blue,mark=triangle*] coordinates { 
(0.141421356237,0.00831635505331)
(0.0707106781187,0.00045142517359)
(0.0353553390593,3.20036659404e-05)
(0.0176776695297,3.71195777312e-06)
(0.00883883476483,4.50234880176e-07)
(0.00441941738242,5.61273676175e-08)
(0.00220970869121,7.01236777319e-09)
 };
 \addplot[color=red,mark=*] coordinates { 
(0.141421356237,0.0221836791887)
(0.0707106781187,0.0034724283226)
(0.0353553390593,0.000731932099675)
(0.0176776695297,0.000176365734205)
(0.00883883476483,4.34240063754e-05)
(0.00441941738242,1.07725987736e-05)
(0.00220970869121,2.6819562071e-06)
 };
\logLogSlopeTriangle{0.43}{0.2}{0.52}{2}{red};
\logLogSlopeTriangle{0.53}{0.2}{0.35}{3}{blue};
 \legend{$\|u-u_h\|_{0,\Omega_h}/\|u\|_{0,\Omega_h}$,$|u-u_h|_{1,\Omega_h}/|u|_{1,\Omega_h}$}
\end{loglogaxis}
\end{tikzpicture}
\begin{tikzpicture}[thick,scale=0.8, every node/.style={scale=1.0}]
\begin{loglogaxis}[xlabel=$h$,xmin=3e-3,xmax=0.3 ,ymin=1e-11
,legend pos=south east
,legend style={ font=\large}
, legend columns=1]
 \addplot[color=blue,mark=triangle*] coordinates { 
(0.141421356237,0.000330381209779)
(0.0707106781187,1.18357229865e-05)
(0.0353553390593,5.86764362779e-07)
(0.0176776695297,3.57242402101e-08)
(0.00883883476483,2.19841962548e-09)
(0.00441941738242,1.37073806632e-10)

 };
 \addplot[color=red,mark=*] coordinates { 
(0.141421356237,0.00156735001333)
(0.0707106781187,9.57193311326e-05)
(0.0353553390593,8.38160094782e-06)
(0.0176776695297,9.66049034669e-07)
(0.00883883476483,1.18062253543e-07)
(0.00441941738242,1.46280911289e-08)
 };
\logLogSlopeTriangle{0.43}{0.2}{0.6}{3}{red};
\logLogSlopeTriangle{0.6}{0.2}{0.35}{4}{blue};
 \legend{$\|u-u_h\|_{0,\Omega_h}/\|u\|_{0,\Omega_h}$,$|u-u_h|_{1,\Omega_h}/|u|_{1,\Omega_h}$}
\end{loglogaxis}
\end{tikzpicture}
\caption{Relative errors of $\phi$-FEM for the test case \eqref{eq:test1}--\eqref{eq:test1e}. Left: $k=2$; Right: $k=3$.  }
\label{fig:simu2}
\end{center}\end{figure}
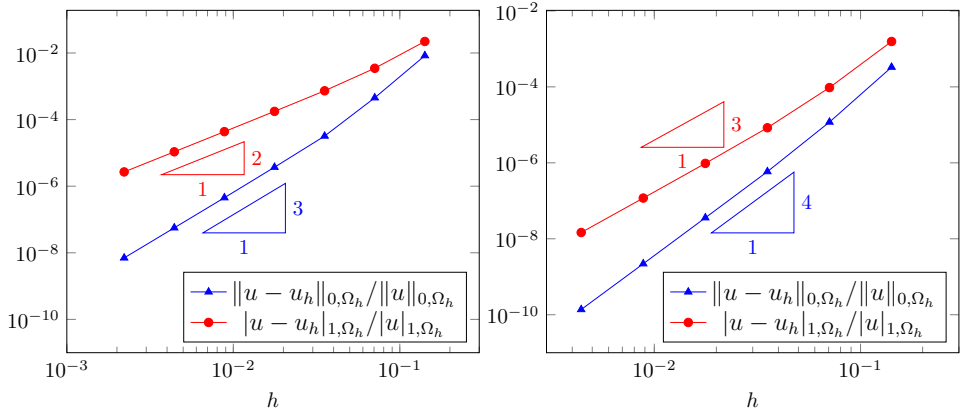

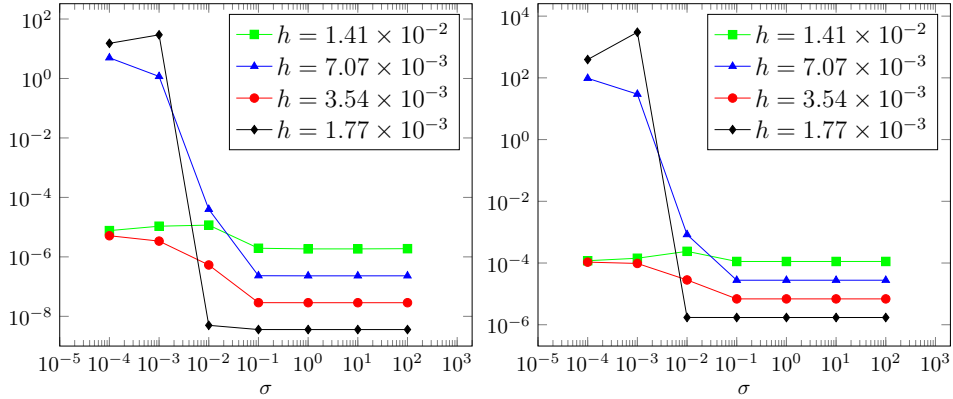
\begin{figure}
\begin{center}
\begin{tikzpicture}[thick,scale=0.8, every node/.style={scale=1.0}]
\begin{loglogaxis}[xlabel=$\sigma$,xmin=1e-5,xmax=2000 ,ymin=1e-9
,legend pos=north east
,legend style={ font=\large}
, legend columns=1]
 \addplot[color=green,mark=square*] coordinates { 
(100.0,1.88752055852e-06)
(10.0,1.85520907983e-06)
(1.0,1.85436200164e-06)
(0.1,1.94844259005e-06)
(0.01,1.17040369772e-05)
(0.001,1.07506410165e-05)
(0.0001,7.6447230241e-06)
 };
 \addplot[color=blue,mark=triangle*] coordinates { 
(100.0,2.30479830979e-07)
(10.0,2.30110431147e-07)
(1.0,2.30129662815e-07)
(0.1,2.32209203179e-07)
(0.01,3.99011709184e-05)
(0.001,1.17194857293)
(0.0001,4.96635451549)
 };
 \addplot[color=red,mark=*] coordinates { 
(100.0,2.87363825038e-08)
(10.0,2.87299557075e-08)
(1.0,2.87316883527e-08)
(0.1,2.88342645869e-08)
(0.01,5.33616420266e-07)
(0.001,3.39184149576e-06)
(0.0001,5.13818431957e-06)
 };
 \addplot[color=black,mark=diamond*] coordinates { 
(100.0,3.59012499276e-09)
(10.0,3.59004160785e-09)
(1.0,3.5900934539e-09)
(0.1,3.59299632591e-09)
(0.01,4.98821933415e-09)
(0.001,29.3894197717)
(0.0001,15.2065983166)

 };
 \legend{
  $h=1.41\times 10^{-2}$, $h=7.07\times 10^{-3}$, $h=3.54\times 10^{-3}$,
 $h=1.77\times 10^{-3}$}
\end{loglogaxis}
\end{tikzpicture}
\begin{tikzpicture}[thick,scale=0.8, every node/.style={scale=1.0}]
\begin{loglogaxis}[xlabel=$\sigma$,xmin=1e-5,xmax=2000
,legend pos=north east
,legend style={ font=\large}
, legend columns=1]
 \addplot[color=green,mark=square*] coordinates { 
(100.0,0.000112340937596)
(10.0,0.000112170206109)
(1.0,0.000112121765036)
(0.1,0.000112128929618)
(0.01,0.000239062039608)
(0.001,0.000142823889172)
(0.0001,0.000118626707906)
 };
 \addplot[color=blue,mark=triangle*] coordinates { 
(100.0,2.77012107524e-05)
(10.0,2.76969457996e-05)
(1.0,2.7695333005e-05)
(0.1,2.76952720528e-05)
(0.01,0.000834997912534)
(0.001,29.6132784023)
(0.0001,96.7740419621)

 };
 \addplot[color=red,mark=*] coordinates { 
(100.0,6.88030187546e-06)
(10.0,6.88015818101e-06)
(1.0,6.88011070214e-06)
(0.1,6.88011201147e-06)
(0.01,2.82763631198e-05)
(0.001,9.6888173371e-05)
(0.0001,0.000106762593875)

 };
  \addplot[color=black,mark=diamond*] coordinates { 
(100.0,1.71510981742e-06)
(10.0,1.71510546277e-06)
(1.0,1.71510396395e-06)
(0.1,1.71510391674e-06)
(0.01,1.72170110121e-06)
(0.001,3011.37723859)
(0.0001,391.901598457)

 };
 \legend{
  $h=1.41\times 10^{-2}$, $h=7.07\times 10^{-3}$, $h=3.54\times 10^{-3}$,
 $h=1.77\times 10^{-3}$}
\end{loglogaxis}
\end{tikzpicture}
\caption{Influence of the ghost penalty parameter $\sigma$ on the relative errors for $\phi$-FEM in the test case \eqref{eq:test1}--\eqref{eq:test1e} and $k=2$. Left: $\|u-u_h\|_{0,\Omega}/\|u\|_{0,\Omega}$; Right: $|u-u_h|_{1,\Omega}/|u|_{1,\Omega}$.  }
\label{fig:sigma2}
\end{center}\end{figure}

\begin{figure}
\begin{center}
\begin{tikzpicture}[thick,scale=0.8, every node/.style={scale=1.0}]
\begin{loglogaxis}[xlabel=$h$,xmin=3e-3,xmax=0.2 ,ymin=1e-6
,legend pos=south east
,legend style={ font=\large}
, legend columns=1]
 \addplot[color=blue,mark=triangle*] coordinates { 
(0.141421356237,0.269401360954)
(0.0707106781187,0.0700712871866)
(0.0353553390593,0.0157512657253)
(0.0176776695297,0.00321769927243)
(0.00883883476483,0.000610890295584)
(0.00441941738242,0.000108054152842)
 };
 \addplot[color=red,mark=*] coordinates { 
(0.141421356237,0.409376356706)
(0.0707106781187,0.172617335932)
(0.0353553390593,0.0677654773812)
(0.0176776695297,0.025037583361)
(0.00883883476483,0.00857797956553)
(0.00441941738242,0.00268500369835)
 };
\logLogSlopeTriangle{0.31}{0.2}{0.53}{1}{red};
\logLogSlopeTriangle{0.43}{0.2}{0.35}{2}{blue};
 \legend{$\|u_{ref}-u_h\|_{0,\Omega}/\|u_{ref}\|_{0,\Omega}$,$|u_{ref}-u_h|_{1,\Omega}/|u_{ref}|_{1,\Omega}$}
\end{loglogaxis}
\end{tikzpicture}
\begin{tikzpicture}[thick,scale=0.8, every node/.style={scale=1.0}]
\begin{loglogaxis}[xlabel=$h$,xmin=3e-3,xmax=0.2 ,ymin=1e-8
,legend pos=south east
,legend style={ font=\large}
, legend columns=1]
 \addplot[color=blue,mark=triangle*] coordinates { 
(0.141421356237,0.00287444838444)
(0.0707106781187,0.000216963066373)
(0.0353553390593,3.82265162572e-05)
(0.0176776695297,3.61312152698e-06)
(0.00883883476483,3.33487641843e-07)
(0.00441941738242,7.40791195515e-08)
 };
 \addplot[color=red,mark=*] coordinates { 
(0.141421356237,0.0126902686837)
(0.0707106781187,0.0037706718527)
(0.0353553390593,0.00141251690317)
(0.0176776695297,0.000240884404215)
(0.00883883476483,5.07259996134e-05)
(0.00441941738242,2.1126114462e-05)
 };
\logLogSlopeTriangle{0.45}{0.2}{0.53}{2}{red};
\logLogSlopeTriangle{0.65}{0.2}{0.35}{3}{blue};
 \legend{$\|u_{ref}-u_h\|_{0,\Omega}/\|u_{ref}\|_{0,\Omega}$,$|u_{ref}-u_h|_{1,\Omega}/|u_{ref}|_{1,\Omega}$}
\end{loglogaxis}
\end{tikzpicture}
\caption{Relative errors of $\phi$-FEM for the test case \eqref{eq:test2}--\eqref{eq:test2e}. Left: $k=1$; Right: $k=2$. The reference solution $u_{ref}$ is computed by a standard $FEM$ on a sufficiently fine fitted mesh on $\Omega$. }
\label{fig:test2 simu1}
\end{center}\end{figure}
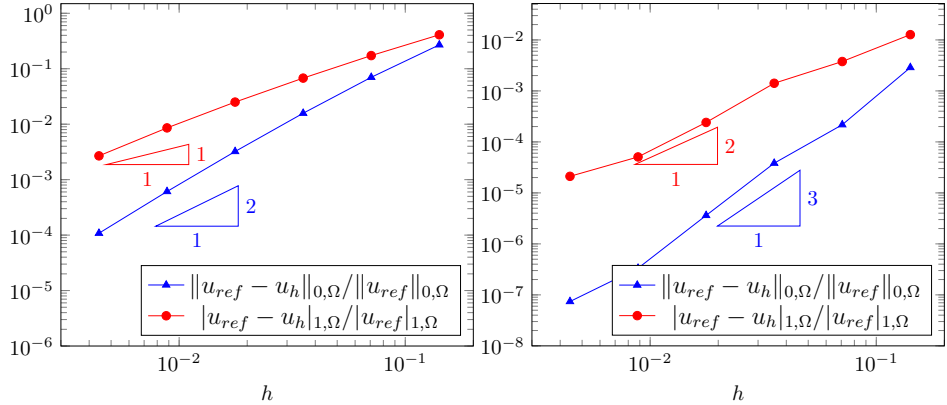

\section{Conclusions}
The numerical results from the last section confirm the theoretically predicted optimal convergence of $\phi$-FEM in the $H^1$ semi-norm. The convergence in the $L^2$ norm turns out to be also optimal, which is better than the theoretical prediction. We have thus an easily implementable optimally convergent finite element method suitable for non-fitted meshes and robust with respect to the cuts of the mesh with the domain boundary. 

Of course, the scope of the present article is very limited and academic: we only consider here the Poisson equation with homogeneous boundary conditions. An extension to non-homogeneous Dirichlet $u=g$ on $\Gamma$ is straightforward if $g$ is given in a vicinity of $\Gamma$: one can the put $u_h=g_h+\phi_hw_h$ with $g_h$ a finite element approximation to $g$ extended by 0 far from $\Gamma$. On the other hand, treating Neumann or Robin boundary conditions would be a completely different matter. We hope that the ideas from \cite{lozinski2018} could be reused under a $\phi$-FEM flavor in this case as well. Future endeavors should then be devoted to more complcated governing equations.   

%

\bibliographystyle{siamplain}
\bibliography{cutfem}

\begin{thebibliography}{10}

\bibitem{LoggMardal2012}
{\sc L.~Anders, M.~Kent-Andre, G.~N. Wells, and al}, {\em Automated Solution of
  Differential Equations by the Finite Element Method}, Springer, 2012,
  \url{https://doi.org/10.1007/978-3-642-23099-8}.

\bibitem{brenner2007mathematical}
{\sc S.~Brenner and L.~Scott}, {\em The mathematical theory of finite element
  methods}, vol.~15 of Texts in Applied Mathematics, Springer, New York,
  third~ed., 2008, \url{https://doi.org/10.1007/978-0-387-75934-0},
  \url{https://doi.org/10.1007/978-0-387-75934-0}.

\bibitem{burmanghost}
{\sc E.~Burman}, {\em Ghost penalty}, Comptes Rendus Mathematique, 348 (2010),
  pp.~1217--1220.

\bibitem{burman15}
{\sc E.~Burman, S.~Claus, P.~Hansbo, M.~G. Larson, and A.~Massing}, {\em
  Cutfem: discretizing geometry and partial differential equations},
  International Journal for Numerical Methods in Engineering, 104 (2015),
  pp.~472--501.

\bibitem{burman1}
{\sc E.~Burman and P.~Hansbo}, {\em Fictitious domain finite element methods
  using cut elements: I. {A} stabilized {L}agrange multiplier method}, Computer
  Methods in Applied Mechanics and Engineering, 199 (2010), pp.~2680--2686.

\bibitem{burman2}
{\sc E.~Burman and P.~Hansbo}, {\em Fictitious domain finite element methods
  using cut elements: Ii. {A} stabilized {N}itsche method}, Applied Numerical
  Mathematics, 62 (2012), pp.~328--341.

\bibitem{burman3}
{\sc E.~Burman and P.~Hansbo}, {\em Fictitious domain methods using cut
  elements: Iii. a stabilized nitsche method for stokes’ problem}, ESAIM:
  Mathematical Modelling and Numerical Analysis, 48 (2014), pp.~859--874.

\bibitem{burman18}
{\sc E.~Burman, P.~Hansbo, and M.~Larson}, {\em A cut finite element method
  with boundary value correction}, Mathematics of Computation, 87 (2018),
  pp.~633--657.

\bibitem{evans}
{\sc L.~Evans}, {\em Partial differential equations}, vol.~19 of Graduate
  Studies in Mathematics, American Mathematical Society, Providence, RI, 1998.

\bibitem{HaslingerRenard}
{\sc J.~Haslinger and Y.~Renard}, {\em A new fictitious domain approach
  inspired by the extended finite element method}, SIAM Journal on Numerical
  Analysis, 47 (2009), pp.~1474--1499.

\bibitem{lozinski2018}
{\sc A.~Lozinski}, {\em A new fictitious domain method: Optimal convergence
  without cut elements}, submitted, preprint {arXiv:1901.03966},  (2019).

\bibitem{moes06}
{\sc N.~Mo{\"e}s, E.~B{\'e}chet, and M.~Tourbier}, {\em Imposing dirichlet
  boundary conditions in the extended finite element method}, International
  Journal for Numerical Methods in Engineering, 67 (2006), pp.~1641--1669.

\bibitem{moes99}
{\sc N.~Mo{\"e}s, J.~Dolbow, and T.~Belytschko}, {\em A finite element method
  for crack growth without remeshing}, International journal for numerical
  methods in engineering, 46 (1999), pp.~131--150.

\bibitem{HardyBlog}
{\sc S.~Montgomery-Smith}, {\em Hardy's inequality for integrals}.
\newblock
  \url{https://math.stackexchange.com/questions/83946/hardys-inequality-for-integrals},
  2011.

\bibitem{osher}
{\sc S.~Osher and R.~Fedkiw}, {\em Level set methods and dynamic implicit
  surfaces}, vol.~153 of Applied Mathematical Sciences, Springer-Verlag, New
  York, 2003, \url{https://doi.org/10.1007/b98879},
  \url{https://doi.org/10.1007/b98879}.

\bibitem{sukumar2001}
{\sc N.~Sukumar, D.~L. Chopp, N.~Mo{\"e}s, and T.~Belytschko}, {\em Modeling
  holes and inclusions by level sets in the extended finite-element method},
  Computer methods in applied mechanics and engineering, 190 (2001),
  pp.~6183--6200.

\end{thebibliography}
\end{document}